\documentclass[11pt]{article}
\usepackage{amssymb}
\usepackage{amsthm}
\usepackage{amsmath}
\usepackage[all]{xy}
\usepackage{hyperref} 

\theoremstyle{definition}

\newtheorem* {theorem*}{Theorem}

\newtheorem{theorem}{Theorem}[section]

\theoremstyle{definition}
\newtheorem{observation}[theorem]{Observation}

\newtheorem* {example*}{Example}

\newtheorem{lemma}[theorem]{Lemma}
\theoremstyle{definition}
\newtheorem{definition}[theorem]{Definition}
\theoremstyle{definition}

\newtheorem{proposition}[theorem]{Proposition}
\newtheorem{corollary}[theorem]{Corollary}
\newtheorem* {remark}{Remark}
\theoremstyle{definition}

\theoremstyle{definition}

\theoremstyle{definition}

\theoremstyle{definition}
\newtheorem* {remarks}{Remarks}

\xyoption{dvips}

\usepackage{fullpage}

\numberwithin{equation}{section}

\def\({\left(}
\def\){\right)}
 \newcommand{\FF}{\mathbb{F}}  \newcommand{\CC}{\mathbb{C}}      
 \newcommand{\cK}{\mathcal{K}}  
  \newcommand{\cS}{\mathcal{S}}\newcommand{\cT}{\mathcal{T}} 
 
\newcommand{\cC}{\mathcal{C}}

\newcommand{\cD}{\mathcal{D}}
\newcommand{\cZ}{\mathcal{Z}}
\def\NN{\mathbb{N}}
    \def\ZZ{\mathbb{Z}}  \def\X{\mathfrak{X}} \def\sl{\mathfrak{sl}}   \def\Res{\mathrm{Res}}       
  \def\wt{\widetilde}

   \newcommand{\rank}{\mathrm{rank}}

\def\X{\mathcal{X}}

\def\barr{\begin{array}}
\def\earr{\end{array}}
\def\ba{\begin{aligned}}
\def\ea{\end{aligned}}
\def\be{\begin{equation}}
\def\ee{\end{equation}}
\def\cS{\mathcal{S}}

\def\UT{\textbf{U}}
\def\cX{\mathcal{X}}
\def\cY{\mathcal{Y}}

\def\sP{\Pi}

\def\omdef{:}%\overset{\mathrm{def}}}

\newcommand\stirl[2]
{\genfrac{\{}{\}}{0pt}{}{#1}{#2}}

\newcommand\stirlstirl[2]
{\left\{ \hspace{-1.5mm}{ \stirl{#1}{#2}} \hspace{-1.5mm}\right\}}

\newcommand\stirlb[2]
{\stirl{#1}{#2}_{\hspace{-0.5mm}\mathrm{B}}\hspace{0.5mm}}

\def\Arc{\mathrm{Arc}}
\def\Cov{\mathrm{Cov}}

\def\X{\mathrm{X}}

\def\B{\mathrm{B}}
\def\D{\mathrm{D}}

\newcommand\Bell[2]{\mathcal{B}_{#1}(#2)}
\newcommand\Cat[2]{\mathcal{N}_{#1}(#2)}
\newcommand\CatB[2]{\mathcal{N}^{\B}_{#1}(#2)}
\newcommand\CatD[2]{\mathcal{N}^{\D}_{#1}(#2)}

\newcommand\BellB[2]{\mathcal{B}^{\B}_{#1}(#2)}
\newcommand\BellD[2]{\mathcal{B}^{\D}_{#1}(#2)}

\newcommand\CatX[2]{\mathcal{N}^{\X}_{#1}(#2)}
\newcommand\BellX[2]{\mathcal{B}^{\X}_{#1}(#2)}

\newcommand\Fe[2]{\mathcal{F}_{#1}(#2)}

\newcommand\Mo[2]{ \mathcal{M}_{#1}(#2)}
\newcommand\MoX[2]{ \mathcal{M}^{\X}_{#1}(#2)}
\newcommand\wMoB[2]{ \widetilde{\mathcal{M}}^{\B}_{#1}(#2)}
\newcommand\MoB[2]{ \mathcal{M}^{\B}_{#1}(#2)}
\newcommand\MoD[2]{ \mathcal{M}^{\D}_{#1}(#2)}

\newcommand\FeX[2]{\mathcal{F}^{\X}_{#1}(#2)}
\newcommand\FeB[2]{\mathcal{F}^{\B}_{#1}(#2)}
\newcommand\wFeB[2]{ \widetilde{\mathcal{F}}^{\B}_{#1}(#2)}
\newcommand\FeD[2]{\mathcal{F}^{\D}_{#1}(#2)}

\def\NC{\mathrm{N  C}}
\def\NCB{\NC^{\B}}
\def\NCD{\NC^{\D}}

\def\tNC{\wt\NC{}}
\def\tNCX{\tNC^\X}
\def\tNCB{\tNC^{\B}}
\def\tNCD{\tNC^{\D}}

\def\NN{\mathrm{NN}}

\def\NNB{\NN^{\B}}

\def\sPX{\sP^{\X}}
\def\sPC{\sP^{\mathrm C}}
\def\sPB{\sP^{\B}}
\def\sPD{\sP^{\D}}

\def\chiB{\chi^\B}
\def\chiD{\chi^\D}

\def\ds{\displaystyle}

\def\shift{\text{\emph{shift}}}
\def\AA{\mathbb{A}}
\def\Lin{\mathrm{L}}

\def\LinX{\mathrm{L}^{\X}}
\def\LinB{\mathrm{L}^{\mathrm{B}}}
\def\LinD{\mathrm{L}^{\mathrm{D}}}

\def\BB{\mathbb{B}}

\makeatletter
\renewcommand{\@makefnmark}{\mbox{\textsuperscript{}}}
\makeatother

\allowdisplaybreaks[1]

\UseCrayolaColors

\title{Actions and identities on set partitions}

\author{Eric Marberg
\\
\small Department of Mathematics\\[-0.8ex]
\small Massachusetts Institute of Technology\\
\small \texttt{emarberg@math.mit.edu}
}
\date{}

\begin{document}

\maketitle

\begin{abstract}
A labeled set partition is a partition of a set of integers whose arcs are labeled by nonzero elements of an abelian group $\AA$.
Inspired by the action of the linear characters of the unitriangular group on its supercharacters, we define a group action of $\AA^n$ on the set of $\AA$-labeled partitions of an $(n+1)$-set.
By investigating the orbit decomposition of various families of set partitions under this action, we derive new combinatorial proofs of Coker's identity for the Narayana polynomial and its type B analogue, and establish a number of other related identities. In return, we also prove some enumerative results concerning Andr\'e and Neto's supercharacter theories of type B and D.
\end{abstract}

%\tableofcontents

\section{Introduction}

A \emph{set partition} is formally a set of nonempty, pairwise disjoint sets, which we always assume consist of integers and which we refer to as \emph{blocks}. % (always assumed to consist of integers) referred to as \emph{blocks}. %We write $\Lambda \vdash \cX$ and say that $\Lambda$ is a \emph{partition} of a set $\cX$ if $\Lambda$ is a set partition the union of whose blocks is $\cX$. 
%Assume $\cX$ is a set of integers.
We call a pair of integers $(i,j)$   an \emph{arc} of a set partition  if $i$ and $j$ occur in the same block and $j$ is the least element of the block greater than $i$.  Let $\Arc(\Lambda)$ denote the set of arcs of a set partition $\Lambda$. %, so that  if  $\Lambda = \{ \{1,3,4\},\{2,6\},\{5\} \}$, for example, then $\Arc(\Lambda) = \{ (1,3),(3,4),(2,6)\}$.

We write $\Lambda \vdash \cX$ and say that $\Lambda$ is a \emph{partition} of a set $\cX$ if $\Lambda$ is a set partition the union of whose blocks is $\cX$. 
The \emph{standard representation} of a  partition $\Lambda\vdash\cX$ is then the directed graph with vertex set $\cX$ and edge set $\Arc(\Lambda)$, drawn by listing the elements of $\cX$ in order with the corresponding arcs overhead.
For example, the set partitions $\Lambda= \{ \{ 1,3,4,7\} ,\{2,6\}, \{5\} \}$ and $\Gamma =  \{ \{ 1,7\} ,\{2,3,4,6\}, \{5\} \} $ have %, $\Cov(\Lambda) = \{ (3,4)\}$,
  standard representations
\be\label{std-ex} \Lambda = \xy<0.0cm,-0.0cm> \xymatrix@R=-0.0cm@C=.3cm{
*{\bullet} \ar @/^.7pc/ @{-} [rr]   & 
*{\bullet} \ar @/^1.0pc/ @{-} [rrrr] &
*{\bullet} \ar @/^.5pc/ @{-} [r]  &
*{\bullet} \ar @/^0.7pc/ @{-} [rrr] &
*{\bullet} &
*{\bullet}  &
*{\bullet}\\
1   & 
2 &
3  &
4 &
5 &
6 & 
7
}\endxy
\qquad\text{and}\qquad
 \Gamma = \xy<0.0cm,-0.0cm> \xymatrix@R=-0.0cm@C=.3cm{
*{\bullet} \ar @/^1.2pc/ @{-} [rrrrrr]   & 
*{\bullet} \ar @/^.5pc/ @{-} [r] &
*{\bullet} \ar @/^.5pc/ @{-} [r]  &
*{\bullet} \ar @/^0.7pc/ @{-} [rr] &
*{\bullet} &
*{\bullet} &
*{\bullet}\\
1   & 
2 &
3  &
4 &
5 &
6 &
7
}\endxy
\ee
since $\Arc(\Lambda) = \{ (1,3), (2,6), (3,4),(4,7)\}$ and $\Arc(\Gamma) = \{ (1,7),(2,3),(3,4),(4,6)\}$.  
A set partition $\Lambda$ is \emph{noncrossing} if no two arcs $(i,k),(j,l) \in \Arc(\Lambda)$ have $i<j<k<l$, which means that no arcs cross in its standard representation.

%

%\begin{definition} \label{labeled-set-part}
This paper investigates a group action on   set  partitions which are labeled in the following sense.
Given  an additive abelian group $\AA$, an \emph{$\AA$-labeled set partition} is a set partition $\Lambda$ with a map $\Arc(\Lambda) \to \AA\setminus\{0\}$, denoted $(i,j) \mapsto \Lambda_{ij}$.
 % \end{definition}
%
%\begin{remark}
This is essentially the definition of a colored rhyming scheme as studied in \cite{Rogers}, except that we require the colors  to form the set of nonzero elements of an abelian group. 
%Our motivation for this convention is that so-defined, the $\FF_q$-labeled partitions of $[n]$ naturally index the \emph{supercharacters} of the unitriangular group $\UT_n(\FF_q)$. This is the group of $n\times n$ unipotent upper triangular matrices over a finite field with $q$ elements; its supercharacters make up a certain family complex characters whose constituents partition the group's irreducible characters and which have a number of other interesting properties (see \cite{Thiem} for a concise overview).
%
%\end{remark}
%Let $\sP(n,\AA)$ and $\NC(n,\AA)$ denote the sets of ordinary and noncrossing $\AA$-labeled  partitions of $[n]$.  
For each nonnegative integer $n$, we define
\[ \ba 
\sP(n,\AA) &\omdef =  \text{the set of $\AA$-labeled  partitions of $[n]\omdef=\{ i \in \ZZ : 1\leq i \leq n\}$}, \\
\NC(n,\AA) &\omdef =  \text{the set of $\AA$-labeled noncrossing partitions of $[n]$}, \\
\Lin(n,\AA) &\omdef =  \text{the set of $\AA$-labeled  partitions of $[n]$  with consecutive integer blocks}.\ea\] 
Note that an $\AA$-labeled partition $\Lambda \vdash[n]$ belongs to 
$\Lin(n,\AA)$ if and only if every arc of $\Lambda$ has the form $(i,i+1)$ for some $i \in [n-1]$. %; thus $|\Lin(n,\AA)| = |\AA|^{n-1}$. 
We begin by defining an operation of $\Lin(n,\AA)$ on $\sP(n,\AA)$.

\begin{definition}\label{+def}
Given $\alpha \in \Lin(n,\AA)$ and $\Lambda \in \sP(n,\AA)$, define $\alpha+\Lambda$ as the $\AA$-labeled partition of $[n]$ whose standard representation is obtained by the following procedure:
\begin{enumerate}
\item[$\bullet$] List the numbers $1,2,\dots,n$  and draw the labeled arcs of both $\alpha$ and $\Lambda$ overhead.  

\item[$\bullet$] Whenever two arcs coincide,  add their labels and replace the pair with a single arc.  

\item[$\bullet$] Whenever two distinct arcs share an endpoint,  delete the shorter arc.  

\item[$\bullet$] Finally, remove any arcs labeled by zero.
\end{enumerate}
%Define $\alpha+\Lambda$ to be the corresponding $\AA$-labeled set partition.  
For example, 
\[ \ba
\( \xy<0.0cm,-0.0cm> \xymatrix@R=-0.0cm@C=.3cm{
*{\bullet} \ar @/^.5pc/ @{-} [r]^{a} &
*{\bullet} \ar @/^.5pc/ @{-} [r]^b  &
*{\bullet}\ar @/^.5pc/ @{-} [r]^c &
*{\bullet} \ar @/^.5pc/ @{-} [r]^d &
*{\bullet} \\
1   & 
2 &
3  &
4 &
5 
}\endxy\)
+
\( \xy<0.0cm,-0.0cm> \xymatrix@R=-0.0cm@C=.3cm{
*{\bullet} \ar @/^.5pc/ @{-} [r]^{-a} &
*{\bullet} \ar @/^.5pc/ @{-} [r]^{e} &
*{\bullet} \ar @/^0.7pc/ @{-} [rr]^f &
*{\bullet} &
*{\bullet} \\
1   & 
2 &
3  &
4 &
5 
}\endxy\)
&=
\xy<0.0cm,-0.0cm> \xymatrix@R=-0.0cm@C=.3cm{
*{\bullet} \ar @/^.5pc/ @{-} [r]^0  &
*{\bullet} \ar @/^.5pc/ @{-} [r]^{b+e}  &
*{\bullet} \ar @/^0.8pc/ @{-} [rr]^f \ar @/^.5pc/ @{-} [r]_c &
*{\bullet} \ar @/^.5pc/ @{-} [r]_d &
*{\bullet} \\
1   & 
2 &
3  &
4 &
5 
}\endxy
\\&
=
\xy<0.0cm,-0.0cm> \xymatrix@R=-0.0cm@C=.3cm{
*{\bullet} \ar @/^.5pc/ @{-} [r]^0  &
*{\bullet} \ar @/^.5pc/ @{-} [r]^{b+e}  &
*{\bullet} \ar @/^0.8pc/ @{-} [rr]^f  &
*{\bullet}  &
*{\bullet} \\
1   & 
2 &
3  &
4 &
5 
}\endxy
\\&
=
\xy<0.0cm,-0.0cm> \xymatrix@R=-0.0cm@C=.3cm{
*{\bullet}  &
*{\bullet} \ar @/^.5pc/ @{-} [r]^{b+e}  &
*{\bullet} \ar @/^0.8pc/ @{-} [rr]^f &
*{\bullet} &
*{\bullet} \\
1   & 
2 &
3  &
4 &
5 
}\endxy
\ea\]
for $a,b,c,d,e,f \in \AA\setminus\{0\}$ with $b\neq -e$.
\end{definition}

The operation $+$ gives $\Lin(n,\AA)$  the structure of an abelian group isomorphic to $\AA^{n-1}$ which acts on both $\sP(n,\AA)$ and $\NC(n,\AA)$.  This action has several interesting properties and serves as a useful tool for providing succinct combinatorial proofs of some notable identities. The following is a motivating example.
In studying some enumerative problems associated with a class of lattice paths, Coker \cite{Coker_Enum}  derived, using generating functions and the Lagrange inversion formula, the equivalence of two expressions for the rank generating function of the lattice of noncrossing partitions of type $A_n$. This amounted to the  identity
\be\label{intro1} \sum_{k=0}^n \frac{1}{n+1} \binom{n+1}{k} \binom{n+1}{k+1} x^k = \sum_{k=0}^{\lfloor n/2 \rfloor} \cC_k \binom{n}{2k} x^k (1+x)^{n-2k},\ee
with $\cC_k = \frac{1}{k+1} \binom{2k}{k}$ denoting the $k$th Catalan number.  
Somewhat earlier, Riordan 
included in his book \cite{riordan} a similar equation involving the rank generating function of the lattice of noncrossing partitions of type $B_n$:
\be\label{intro2} \sum_{k=0}^n \binom{n}{k}^2 x^k = \sum_{k=0}^{\lfloor n/2\rfloor}  \binom{2k}{k}  \binom{n}{2k} x^k (x+1)^{n-2k}.
\ee
%The action of $\Lin(n,\AA)$ on $\sP(n,\AA)$ provides a short, alternate  proof of these identities. In particular, we obtain (\ref{intro1}) by noting that when $x=|\AA|-1$, the left  side counts the elements of $\NC(n,\AA)$ according to their numbers of  blocks 
We obtain a simple combinatorial proof of (\ref{intro1}) from Definition \ref{+def} by noting that  when $x=|\AA|-1$, the terms in the left  sum count the elements of $\NC(n+1,\AA)$ with $n-k$ blocks, while the terms in the right  sum count the elements of $\NC(n+1,\AA)$ whose $\Lin(n+1,\AA)$-orbits have size $|\AA|^{n-2k}$. The second identity    (\ref{intro2}) follows by the same argument applied to a certain family of  ``type B'' $\AA$-labeled set partitions; see the remarks to Theorems \ref{identities} and \ref{b-identities} below.

\newcommand*\pFq[2]{{}_{#1}F_{#2}\genfrac[]{0pt}{}}
\newcommand*\F[4]{{}_{2}F_{1}\hspace{-0.5mm}\(#1,#2;#3;#4\)}

\begin{remark}
Two recent papers have supplied combinatorial proofs for (\ref{intro1}) and (\ref{intro2}) using  quite different methods.  In \cite{CYY}, Chen, Yan, and Yang prove (\ref{intro1}) by inspecting a weighted version of a bijection between Dyck paths and 2-Motzkin paths; in \cite{CWZ}, Chen, Wang, and Zhang prove (\ref{intro2}) by enumerating certain weighted type B noncrossing partitions.
Algebraic proofs of these identities are much easier to come by: as pointed out by Christian Krattenthaler,
 (\ref{intro1}) and (\ref{intro2}) are respectively
the  special cases 
$(a,b) = (-n,-n-1)$ and $(a,b) = (-n,-n)$ of 
%$i=1$ and $i=0$ of the identity
%\be \sum_{k=0}^n \binom{k+i}{k}^{-1}\binom{n}{k} \binom{n+i}{k} x^k = \sum_{k=0}^{\lfloor n/2\rfloor} \binom{k+i}{k}^{-1} \binom{2k}{k} \binom{n}{2k}  x^k (1+x)^{n-2k},\ee which is itself the special case 
%$(a,b) = (-n,-n-i)$ of
 the quadratic transformation formula
\be\label{hg} \F{a}{b}{1+a-b}{x} = \frac{\F{\tfrac{a}{2}}{\tfrac{a}{2} + \tfrac{1}{2}}{1+a-b}{ \tfrac{4x}{(1+x)^2}}}{(1+x)^{a}}
\ee  for the hypergeometric function \cite[Eq.\ 2.11(34)]{Erdelyi}. This more general identity has been  known since at least 1881, when it appeared in an equivalent form as \cite[Eq.\ (36)]{Goursat}.
\end{remark}

Definition \ref{+def} is motivated by the representation theory of $\UT_n(\FF_q)$,  the group  of $n\times n$ unipotent upper triangular matrices over a finite field with $q$ elements.
In detail, the $\FF_q$-labeled partitions of $[n]$ naturally index the \emph{supercharacters} of $\UT_n(\FF_q)$, a certain family of complex characters whose constituents partition the set of the group's irreducible characters and which have a number of other notable properties (see \cite{Thiem} for a concise overview).
Given $\lambda \in \sP(n,\FF_q)$, let $\chi_\lambda$ denote the associated supercharacter (see Section \ref{last-sect} below for an explicit definition).
The correspondence $\lambda \mapsto \chi_\lambda$ then defines a bijection from $\Lin(n,\FF_q)$ to the set of linear characters of $\UT_n(\FF_q)$, and if $\alpha \in \Lin(n,\FF_q)$ and $\lambda \in \sP(n,\FF_q)$ then the product of the characters $\chi_\alpha$ and $\chi_\lambda$ is precisely $\chi_{\alpha+\lambda}$ (see \cite[Corollary 4.7]{Thiem}). If general, if $\lambda,\mu \in \sP(n,\FF_q)$ then the product of $\chi_\lambda$ and $\chi_\mu$ is a linear combination $\sum_{\nu \in \sP(n,\FF_q)} c_{\lambda\mu}^\nu \chi_\nu$ for some nonnegative integers $c_{\lambda\mu}^\nu$. Finding a combinatorial rule to determine these coefficients  is an open problem, notably studied in \cite{tensor}.

We organize this article as follows.
In Section \ref{action-sec} we reexamine Definition \ref{+def} in slightly greater detail and introduce a few useful conventions.
We carry out a careful analysis of the labeled set partition orbits  under our action  in Section \ref{classical}, and use this to give combinatorial proofs of several identities in the style of (\ref{intro1}) and (\ref{intro2}). Sections \ref{nonnesting} and 
\ref{b-section}  introduce type B and D analogues for the family of labeled set partitions studied in Section \ref{classical}, and in Section \ref{last} we undertake a similar orbit analysis to prove  analogues of our classical identities in these other types.  In Section \ref{last-sect}, we explore the connection between our methods and the supercharacters of the unitriangular group more closely.  
In particular, we provide explicit, succinct definitions of Andr\'e and Neto's supercharacters of type B and D, and compute the sizes of several natural families of these characters. 
%\begin{notation}
%In what follows, we adopt the convention that summations $\sum_k$ with unspecified bounds mean the same as summations $\sum_{k\in \ZZ}$ over the set of all integers.  We write $\delta_{ij}$ for the Kronecker delta function, and evaluate the binomial coefficient $\binom{n}{k}$ to be zero if $k <0$ or $n < k$.  
%\end{notation}

%
%\begin{notation} Throughout, we adopt the following notational conventions:
%\begin{enumerate}
%%\item[] $e_{ij}$ denotes the elementary $n\times n$ matrix with a 1 in position $(i,j)$ and zeros elsewhere. 
%%\item[] $e_{ij}^*$ denotes the linear functional on $n\times n$ matrices given by $e_{ij}^*(X) = X_{ij}$.
%
%\item[] $\delta_{i}$ denotes the Kronecker delta function.
%
%\item[] $[n]$ denotes the set of the first $n$ positive integers, with $[0] = \varnothing$.
%
%\item[] Coefficients $\binom{n}{k}$, $\stirl{n}{k}$,  $\stirlstirl{n}{k}$, $N(n,k)$ are evaluated as zero if $k<0$ or $k>n$.  
%
%
%\end{enumerate}
%\end{notation}

%\subsection*{Acknowledgements}
%
%I  am grateful to Christian Krattenthaler for pointing out the connection between (\ref{intro1}) and (\ref{intro2}) and the hypergeometric identity (\ref{hg}).

\section{Two equivalent definitions}
\label{action-sec}

In this preliminary section we note two equivalent characterizations of the operation $+$ presented in Definition \ref{+def}, and show how this operation leads to another proof of the rank symmetry of the lattice of noncrossing partitions.  We begin with the following observation, whose derivation from Definition \ref{+def} is a straightforward exercise.

\begin{observation} \label{+obs}
Given $\alpha \in \Lin(n,\AA)$ and $\Lambda \in \sP(n,\AA)$, let 
\[ \ba \cS &= \{ (j,j+1) \in \Arc(\alpha) \cap \Arc(\Lambda) : \alpha_{j,j+1} + \Lambda_{j,j+1} = 0\},\\
\cT &= \{ (j,j+1) \in \Arc(\alpha) : (i,j+1) \notin \Arc(\Lambda) \text{ and }(j,\ell) \notin \Arc(\Lambda)\text{ for all }i,\ell\}.\ea\] Then $\alpha+\Lambda$ is the element of $\sP(n,\AA)$ with arc set $(\Arc(\Lambda) - \cS) \cup \cT$ and labeling map 
\[ (\alpha+\Lambda)_{jk} = \begin{cases} \alpha_{jk} + \Lambda_{jk},&\text{if }(j,k) \in \Arc(\alpha) \cap \Arc(\Lambda) - \cS, \\
\alpha_{jk}, &\text{if }(j,k) \in \cT, \\
\Lambda_{jk},&\text{otherwise}.\end{cases}
\]
\end{observation}

This observation makes clear that 
$\alpha+\Lambda$ differs from $\Lambda$ only in its arcs of the form $(j,j+1)$.  Such arcs are never involved in crossings, and so
the action of $\Lin(n,\AA)$ on $\sP(n,\AA)$ preserves $\NC(n,\AA)$; i.e., $\alpha+\Lambda \in \NC(n,\AA)$ for all $\alpha \in \Lin(n,\AA)$ and $\Lambda \in \NC(n,\AA)$.

There is a useful bijection from $\sP(n,\AA)$ to the set of $n\times n$ matrices over $\AA$ 
which are strictly upper triangular and have at most one nonzero entry in each row and column.  The \emph{matrix} or \emph{rook diagram} of $\Lambda \in \sP(n,\AA)$ is the $n\times n$ matrix with the entry $\Lambda_{ij}$ in position $(i,j)$ for each $(i,j) \in \Arc(\Lambda)$ and zeros elsewhere.  A set partition is  noncrossing if and only if there is no position above the diagonal in its associated matrix which is both strictly south of a nonzero entry in the same column and strictly west of a nonzero entry in the same row.  Likewise, a set partition belongs to $\Lin(n,\AA)$ if and only if its matrix has nonzero positions only on the superdiagonal $\{ (i,i+1) : i \in [n-1]\}$.
This fact shows that we may also equivalently define the operation $+$ as follows:

\begin{observation}\label{+def2}
If $\alpha \in \Lin(n,\AA)$ and $\Lambda \in \sP(n,\AA)$ then $\alpha+\Lambda$ is the element of $\sP(n,\AA)$
 produced by the following procedure:
\begin{enumerate}
\item[1.]  Add the matrices of $\alpha$ and $\Lambda$ to form a matrix $M$ over $\AA$.
\item[2.] Replace with zero any nonzero positions $(i,i+1)$ on the superdiagonal of $M$  which lie strictly below a nonzero position in the same column or strictly to the left of a nonzero position in the same row.
%
% Let $j$ range over the integers $1,2,\dots,n-1$ in that order, and for each $j$,
%%\begin{enumerate}
%%\item[] 
%successively replace with zero any nonzero positions in $M$ on the diagonal $\{ (i,i+j) : i \in [n-j]\}$  which lie strictly below a nonzero position in the same column or strictly to the left of a nonzero position in the same row.
%\end{enumerate}

\item[3.] Define $\alpha + \Lambda$ to be the element of $\sP(n,\AA)$ associated to the modified matrix $M$.
\end{enumerate}
\end{observation}

This formulation of Definition \ref{+def} most clearly illustrates that the addition $+$  makes $\Lin(n,\AA)$ into an abelian group acting on $\sP(n,\AA)$: the group is just the additive group of $n\times n$ matrices over $\AA$ whose nonzero entries are all on the superdiagonal.

Let $\sP(n)$ and $\NC(n)$ denote the sets of ordinary and noncrossing (unlabeled) partitions of $[n]$.  We may  view the elements of $\sP(n)$ and $\NC(n)$ as $\AA$-labeled set partitions by taking $\AA=\FF_2$ to be a finite field  with two elements.
 These sets are partially ordered by \emph{refinement}: $\Gamma \leq \Lambda$ if each block of $\Gamma \vdash[n]$ is contained in some block of $\Lambda \vdash[n]$. This partial order makes  $\sP(n)$ and $\NC(n)$ into graded lattices with height $n-1$ according to the rank function $\rank(\Lambda) = n-|\Lambda|$. The lattice $\NC(n)$ in particular  has a number of remarkable properties and an extensive literature (see \cite{memoir} for a survey).  

As a first application of the action of $\Lin(n,\AA)$ on $\sP(n,\AA)$, we note that the map \[\Lambda \mapsto \{\{1,2,\dots,n\}\} +\Lambda\] defines an involution of the set of ($\FF_2$-labeled) partitions of $[n]$.
We denote the image of $\Lambda \vdash[n]$ under this involution by $\Lambda^+$; the latter partition has the following 
explicit definition, which makes sense even for partitions of sets other than $[n]$.

\begin{definition}
Given a set $\cX \subset \ZZ$ and a  partition $\Lambda \vdash\cX$, let 
$\Lambda^+$ be the partition of $\cX$ with arc set  $(\Arc(\Lambda) - \cS) \cup \cT$, where $\cS = \{(i,i+1) : i \in \ZZ\}$  and $\cT$ is the set of pairs $(i,i+1)$ in $\cX\times \cX$ with the property that $i$ and $i+1$ are respectively  maximal and minimal in their blocks of $\Lambda$.  
\end{definition}

For example, we have  $\{\{1,4,5\},\{2,3\},\{6,7\},\{8\}\}^+ = \{\{14\},\{2\},\{3\},\{5,6\},\{7,8\}\}$.  
Our main point in presenting this involution is simply to note that it gives another proof of the fact that the lattice $\NC(n)$ is rank symmetric.

\begin{proposition}\label{alpha-prop} The map $\Lambda \mapsto \Lambda^+$ is rank inverting on $\NC(n)$. That is,
if $\Lambda \in \NC(n)$ has $k$ blocks, then $\Lambda^+ \in \NC(n)$ has $n+1-k$ blocks. 
\end{proposition}

The involution is not order reversing on $\NC(n)$, since for example
 \[\{\{1\},\{2,3\},\{4\}\}^+ =\{\{1,2\},\{3,4\}\} \not> \{\{1,4\},\{2\},\{3\}\} = \{\{1,4\},\{2,3\}\}^+.\]
Also, the property $|\Lambda^+| = n+1 - |\Lambda|$ may fail if $\Lambda \vdash[n]$ is not noncrossing.

\begin{proof} 
Fix $\Lambda \in \NC(n)$ with $k$ blocks.  If $\{n\}$ is a block of $\Lambda$ and $\Lambda' \in \NC(n-1)$ is the set partition formed by removing this block, then $\Lambda^+$ is given by adding $n$ to the block of $n-1$ in $(\Lambda')^+$.  Hence the number of blocks of $\Lambda^+$ is the same as the number of blocks of $(\Lambda')^+$, which by induction is $(n-1) + 1 -(k-1) = n+1-k$ as desired.  

Suppose the block of $n$ in $\Lambda$ has more than one element.  In this block, $n$ is the largest element; let $m$ be the second largest so that $(m,n) \in \Arc(\Lambda)$.  Let $A $ be the noncrossing partition of $[m]$ with arc set $\Arc(\Lambda) \cap \{ (i,j) : 1\leq i<j\leq m\}$ and let $B$ be the noncrossing  partition of $[n-1]\setminus [m]$ with arc set $\Arc(\Lambda) \cap \{(i,j) : m <i<j<n\}$.  Observe that $A\cup B$ is then a  partition of $[n-1]$, and that because $\Lambda$ is noncrossing, adding $n$ to the block of $m$ in $A\cup B$ recovers $\Lambda$.  

If $m=n-1$ then $|A|=k$ and $B = \varnothing$ and $\Lambda^+ = A^+\cup \{ \{n\}\}$,  so   the number of blocks of $\Lambda^+$ is $1+|A^+| = 1 + (n-1)+1 - |A| = n+1-k$ by induction.  Alternatively, if $m<n-1$ then 
$\Lambda^+$ is formed by adding $n$ to the block of $m$ in $A^+ \cup B^+$.  Thus the number of blocks of $\Lambda^+$ is $|A^+ \cup B^+| = m+1-|A| + (n-1-m) +1 -|B| = n+1-|A|-|B| = n+1-k$, again by induction.
\end{proof}

%Deducing the following proposition from the formula (\ref{added-formula})  is a routine exercise.
%
%\begin{proposition}
%If $\alpha \in \Lin(n,\FF_q)$ and $\Lambda \in \sP(n,\FF_q)$ then $\chi_\alpha\otimes \chi_\Lambda = \chi_{\alpha + \Lambda}$.
%\end{proposition}

%More generally, one can show using (???) that $\chi_{\Lambda+\Gamma}$ is at least a constituent of the product $\chi_\Lambda \otimes \chi_\Gamma$ for all $\Lambda,\Gamma \in \sP(n,\FF_q)$.

\section{Identities for classical set partitions}
\label{classical}

In this section we examine the action of $\Lin(n,\AA)$ on $\sP(n,\AA)$ and $\NC(n,\AA)$ in greater detail. To begin, we note that   
shifting the matrix of a set partition one column to the right corresponds to an injective map 
\[\label{shift}\shift : \sP(n,\AA) \to \sP(n+1,\AA)\] which assigns $\Lambda \in \sP(n,\AA)$ to the  $\AA$-labeled set partition  
of $[n+1]$ with arc set  $\{ (i,j+1) : (i,j) \in \Arc(\Lambda)\}$ and labeling map $(i,j+1) \mapsto \Lambda_{ij}$.  For example,
\[ \shift\( \xy<0.0cm,-0.0cm> \xymatrix@R=-0.0cm@C=.3cm{
*{\bullet} \ar @/^.5pc/ @{-} [r]^{a} &
*{\bullet} \ar @/^.7pc/ @{-} [rr]^{b}  &
*{\bullet} \ar @/^0.7pc/ @{-} [rr]^c &
*{\bullet} &
*{\bullet} \\
1   & 
2 &
3  &
4 &
5 
}\endxy\)
=
 \xy<0.0cm,-0.0cm> \xymatrix@R=-0.0cm@C=.3cm{
*{\bullet} \ar @/^.7pc/ @{-} [rr]^{a} &
*{\bullet} \ar @/^.9pc/ @{-} [rrr]^{b}  &
*{\bullet} \ar @/^0.9pc/ @{-} [rrr]^c &
*{\bullet} &
*{\bullet} &
*{\bullet} \\
1   & 
2 &
3  &
4 &
5 &
6}\endxy
\] for $a,b,c \in \AA\setminus \{0\}$.
The map $\shift$ increases the number of blocks of a set partition by one, and  its image consists of all $\AA$-labeled  partitions of $[n+1]$ with no blocks containing both $i$ and $i+1$ for some  $i \in [n]$; following \cite{reduction1}, we call such set partitions \emph{2-regular}.
 The right inverse of $\shift$ (defined on the matrix of a set partition by deleting the first column and last row then setting all diagonal entries to zero)
%\[ \shiftl : \sP(n+1,\AA) \to \sP(n,\AA)\] which assigns $\Lambda \in \sP(n+1,\AA)$ to the unique $\AA$-labeled set partition of $[n]$ with arc set $\{ (i,j-1) : (i,j) \in \Arc(\Lambda)\} \setminus \{ (i,i) : i \in [n]\}$ and labeling map $(i,j-1) \mapsto \Lambda_{ij}$.  
%The map $\shiftl$
 is precisely the ``reduction algorithm'' presented in \cite{reduction1}; see also \cite{reduction2}.

We say that a set partition of $[n]$ is \emph{feasible} if each of its blocks has at least two elements and \emph{poor} if each of its blocks has at most two elements.  The matrix of a feasible set partition has a nonzero entry in either the $i$th row or $i$th column for each $i \in [n]$, while the matrix of a poor set partition never has a nonzero entry in both the $i$th row and $i$th column.
From these considerations, it is straightforward to deduce the following lemma.

\begin{lemma} \label{shift-bij}
The following restrictions of $\shift$ are bijections:
\begin{enumerate}
\item[(1)] $\{ \text{Feasible  partitions of $[n]$} \} \to \left\{ \ba &\text{2-regular  partitions $\Lambda \vdash [n+1]$ such that} \\&\text{$1+\max B \neq \min B'$ for all blocks $B,B' \in \Lambda$}\ea \right\}$.

\item[(2)] $\{ \text{Poor noncrossing  partitions of $[n]$} \}  \to \{ \text{2-regular noncrossing  partitions of $[n+1]$}\}.$

\end{enumerate}
\end{lemma}

%\begin{remark} The  image of the map in (1) consists of precisely those partitions whose matrices have no nonzero entries on the superdiagonal and which have the property that for each $j \in [n-1]$, the ``superdiagonal hook'' $\{ (i,j) : i<j\} \cup \{ (j,k) : j<k\}$ contains at least one nonzero entry. That the map in (2) is a bijection is essentially what Chen, Deng, and Du check in the proof of \cite[Theorem 3.1]{reduction1}.  As noted there, since $\shift(\Lambda)$ has one more block than $\Lambda$, this bijection explains the observation  of Simion and Ullman \cite{lattice} and Klazar \cite{Klazar} that the number of 2-regular noncrossing partitions of $[n]$ with $k$ blocks is also the number of poor noncrossing partitions of $[n-1]$ with $k-1$ blocks. \end{remark}

To count feasible and poor set partitions, we introduce the polynomials
\[  \Fe{n}{x} \omdef = 
\sum_{\text{Feasible }\Lambda \in \sP(n)} x^{|\Arc(\Lambda)|}
%\sum_{k=0}^n \stirlstirl{n}{k} x^{n-k}
\qquad\text{and}\qquad
\Mo{n}{x} \omdef =\sum_{\text{Poor }\Lambda \in \NC(n)} x^{|\Arc(\Lambda)|}
% \sum_{k=0}^{\lfloor n/2 \rfloor} \cC_{k} \binom{n}{2k} x^{k},
.\]
We can write $\Fe{n}{x} = \sum_{k=0}^n \stirlstirl{n}{k} x^{n-k}$ where 
 $\stirlstirl{n}{k}$ is the \emph{associated Stirling number of the second kind} \cite{Comtet}, which counts the number of  feasible set partitions  of $[n]$ with $k$ blocks and is listed as sequence  \cite[A008299]{OEIS}. The numbers $\Fe{n}{1}$ themselves give sequence \cite[A000296]{OEIS}.
 
 Letting $\cC_n \omdef=  \frac{1}{n+1} \binom{2n}{n}$ as in the introduction,
 we have the following slightly more explicit expression 
for $\Mo{n}{x}$. The functions $\Mo{n}{x}$ are  sometimes called {Motzkin polynomials} (see \cite[A055151]{OEIS}) and  $\Mo{n}{1}$ is the $n$th Motzkin number \cite[A001006]{OEIS}.  
% ; they satisfy the recurrence  
%\[\stirlstirl{n+1}{ k+1}=(k+1) \stirlstirl{n}{ k+1}+n \stirlstirl{n-1}{ k},\quad\text{for }n,k\geq 0,\text{ with }\stirlstirl{i}{0} = \stirlstirl{0}{i} = \delta_{i}.
%\]  
% %, which we recall has the formula $\cC_n = \frac{1}{n+1} \binom{2n}{n}$ for $n\geq 0$.
%As in the introduction, let $\cC_n \omdef=  \frac{1}{n+1} \binom{2n}{n}$ denote the $n$th Catalan number.

\begin{proposition}\label{motzkin-obs} $\ds\Mo{n}{x}=\sum_{k=0}^{\lfloor n/2 \rfloor} \cC_{k} \binom{n}{2k} x^{k}$ for $n\geq 0$.

%For $n\geq 0$, the following identities hold:
%\begin{enumerate}
%\item[(1)] $\ds\Fe{n}{x} = \sum_{k=0}^n \stirlstirl{n}{k} x^{n-k}$.
%\item[(2)] $\ds\Mo{n}{x}=\sum_{k=0}^{\lfloor n/2 \rfloor} \cC_{k} \binom{n}{2k} x^{k}$.
%\end{enumerate}
\end{proposition}

%\begin{remark} The functions $\Mo{n}{x}$ are  sometimes called {Motzkin polynomials} (see \cite[A055151]{OEIS}) and  $\Mo{n}{1}$ is the $n$th Motzkin number \cite[A001006]{OEIS}.  \end{remark}

\begin{proof}
Poor noncrossing set partitions of $[n]$ with $k$ arcs %(i.e., noncrossing set partitions of $[n]$ with $k$ blocks of size two and $n-2k$ blocks of size one)
 are  in bijection with pairs $(S,\Lambda)$, where $S$ is a $2k$-subset of $[n]$ and  $\Lambda$ is a noncrossing partition of  $[2k]$ with $k$ blocks of size two.  There are $\binom{n}{2k}$ choices for $S$ and $\cC_k$ choices for $\Lambda$ (see \cite[Exercise 6.19o]{Stanley}).
 % Less directly, one could also deduce the second part by applying Lemma \ref{shift-bij} to \cite[Theorem 2.7]{Klazar}. 
 %
 %  and \cite[Theorems 3.1]{reduction1}.  To prove part (2) directly, it suffices to show that $\cC_k\binom{n}{2k}$ is the number of  noncrossing set partitions of $[n]$ with $k$ parts of size two and $n-2k$  parts of size one.  This clearly holds if and only if $\cC_k$ is the number of noncrossing set partitions of $[2k]$ with $k$ parts of size two; i.e., 
\end{proof}

To state the main theorem of this section, we require one last piece of notation.  Define $\Cov(\Lambda)$ for a set partition $\Lambda$ to be the set of arcs $(i,j) \in \Arc(\Lambda)$ with $j=i+1$; in other words, 
\[\Cov(\Lambda)  = \Arc(\Lambda) \cap \{(i,i+1) : i \in \ZZ\}.\]
%
%For posterity we  note the following brief application of this definition.
%Write $\parts{n}{k}$ for the number of distinguished blocks with size $k$ in all set partitions of $[n]$. By \emph{distinguished}, we mean that we consider two blocks to be equal if and only if they are equal as sets and belong to the same set partition.  %(By even/odd, we mean having an even/odd number of elements.) 
%For example, the set partitions of $\{1,2,3\}$ (in abbreviated notation) are $1|2|3$, $12|3$, $13|2$, $1|23$, and $123$, and they contain 10 distinguished blocks.  
%
%\begin{proposition}
%Fix  integers $n\geq 0$ and $k>0$.  Then $\parts{n}{k}$ is equal to the number of set partitions $\Lambda \vdash[n+1]$ with $|\Cov(\Lambda)| = k$.
%\end{proposition}
%
%
%\begin{proof}
%Let $\cX$ denote the set of pairs $(S,\Lambda)$, where $S$ is a $k$-element subset of $[n]$ and $\Lambda$ is a set partition of the complement of $S$ in $[n]$, and let $\cY  =\{\Lambda \vdash[n+1] : |\Cov(\Lambda)| = k\}$. %, and we will prove (1) by constructing a bijection $\cY\to \cX$.
%Given $\Lambda \in \cY$, let
%$r(\Lambda) = \{ i : (i,i+1) \in \Cov(\Lambda)\}$ and define $\Lambda' $ as the unique set partition of  $[n]\setminus r(\Lambda)$ with the property that $(i,j) \in \Arc(\Lambda')$ if and only if $(i,j+1) \in \Arc(\Lambda)\setminus \Cov(\Lambda)$.
%The map $\Lambda \mapsto (r(\Lambda), \Lambda')$ is a bijection $\cY \to \cX$, and $|\cX| = \parts{n}{k}$ by definition.
%\end{proof}
%
%
%
%
Fix two  additive abelian groups $\AA$ and $\BB$, and  let  $\sP(n,\AA,\BB)$ denote the set of labeled partitions $\Lambda  \in \sP(n,\AA \oplus \BB)$ with 
\be\label{ab-cond}\Lambda_{ij} \in \begin{cases} \AA\setminus\{0\},&\text{if }(i,j) \in \Arc(\Lambda)\setminus \Cov(\Lambda), \\
\BB\setminus\{0\},&\text{if }(i,j) \in \Cov(\Lambda).\end{cases}\ee
%\[  \sP(n,\AA,\BB) \omdef=\text{the set of $\Lambda \in \sP(n,\AA \oplus \BB)$ with $\Lambda_{ij} \in \begin{cases} \AA\setminus\{0\},&\text{if }(i,j) \in \Arc(\Lambda)\setminus \Cov(\Lambda), \\\BB\setminus\{0\},&\text{if }(i,j) \in \Cov(\Lambda),\end{cases}$}
%\\
%\NC(n,\AA,\BB) &\omdef=\text{the set of noncrossing elements of $\sP(n,\AA,\BB)$.}
%\ea\]
We define $\NC(n,\AA,\BB)$ analogously, as the set of noncrossing elements of $\sP(n,\AA,\BB)$.
%We define $\NC(n,\AA,\BB)$, $\Lin(n,\AA,\BB)$, and $\C(n,\AA,\BB)$ analogously.  The groups
Note that the group $\Lin(n,\BB)$ acts on these sets by $+$, and that we may view $\shift$ 
as 
a map $\sP(n,\AA) \to \sP(n+1,\AA,\BB)$.

When  $|\AA| =x+1$ and $ |\BB|=y+1$, the cardinalities of $\sP(n,\AA,\BB)$ and $\NC(n,\AA,\BB)$ are given respectively by the polynomials
\[ \Bell{n}{x,y}  \omdef= \sum_{\Lambda \in \sP(n)} x^{|\Arc(\Lambda)\setminus \Cov(\Lambda)|} y^{|\Cov(\Lambda)|} % |\sP(n,\AA,\BB)|
\quad\text{and}\quad \Cat{n}{x,y}  \omdef=  \sum_{\Lambda \in \NC(n)} x^{|\Arc(\Lambda)\setminus \Cov(\Lambda)|} y^{|\Cov(\Lambda)|}
%|\NC(n,\AA,\BB)|
.\]  
%\[ \Bell{n}{x,y} = \sum_{\Lambda \vdash[n]} x^{|\Arc(\Lambda) \setminus \Cov(\Lambda)|} y^{|\Cov(\Lambda)|}.\]
We will derive more explicit expressions for these functions in a moment.
In the mean time, let $\Bell{n}{x} \omdef= \Bell{n}{x,x}$ and $\Cat{n}{x} \omdef= \Cat{n}{x,x}$.  These simpler polynomials have the formulas \[\Bell{n}{x}  =
\sum_{k=0}^n \stirl{n}{k} x^{n-k}\quad\text{and}\quad 
\Cat{n}{x}   = \sum_{k=0}^n N(n,k) x^{n-k},\]  where $\stirl{n}{k}$ and $N(n,k)$ are the {Stirling numbers of the second kind} and the {Narayana numbers}, defined as the number of ordinary and noncrossing set partitions of $[n]$ with $k$ blocks (equivalently, with $n-k$ arcs). We note the well-known formula $N(n,k) = \frac{1}{n} \binom{n}{k} \binom{n}{k-1}$ for $n> 0$ and adopt the convention $\stirl{0}{k} = N(0,k) = \delta_{k}$. 

\def\BINOMIAL{\mathrm{BINOMIAL}}
\def\INVERT{\mathrm{INVERT}}
\def\R{\mathrm{R}}

\begin{remarks}
Of course, $\Bell{n}{1}$ is the $n$th {Bell number}  and $\Cat{n}{x}$ is  the \emph{Narayana polynomial}, whose values give
 the Catalan numbers    when $x=1$ and the 
little Schr\"oder numbers 
when $x=2$.  The polynomials $\Bell{n}{x}$ and $\Cat{n}{x}$ are by definition  the rank generating functions of the graded lattices $\sP(n)$ and $\NC(n)$, though they have several alternate interpretations:

\begin{itemize}

\item[(i)]  As noted in \cite{BerSlo,Coker_Fam}, when $x$ is a positive integer,  $\{ \Bell{n}{x} \}_{n=0}^\infty$ and $\{ \Cat{n}{x} \}_{n=0}^\infty$ are the unique sequences respectively fixed by the operators 
\[\R\circ \underbrace{\BINOMIAL \circ \cdots\circ  \BINOMIAL}_{x\text{ times}}\qquad\text{and}\qquad\R\circ \underbrace{\INVERT \circ \cdots\circ  \INVERT}_{x\text{ times}},\] where $\R(a_0,a_1,a_2,\dots) \omdef= (1,a_0,a_1,a_2,\dots)$ and $\BINOMIAL$ and $\INVERT$ are the sequence operators defined in \cite{BerSlo}. 

\item[(ii)]  As mentioned in the introduction, the $\FF_q$-labeled partitions of $[n]$ index the supercharacters of  the unitriangular group $\UT_n(\FF_q)$. %, a certain family of reducible complex characters.  
There are  $\Bell{n}{q-1}$ distinct supercharacters of $\UT_n(\FF_q)$, of which $\Cat{n}{q-1}$ are irreducible; see   Section \ref{last-sect}.

\item[(iii)] If $G$ is a finite group then $\Bell{2k}{|G|}$ is the dimension of the $G$-colored partition algebra $P_k(x;G)$ defined in \cite{Bloss} and studied (in the case $G = \ZZ/r\ZZ$) in \cite{Orellana}.

%\item[(iv)] The name ``Narayana polynomial'' seems to have originated in Bonin, Shapiro, and Simion's paper \cite{BSS}, where it actually refers to what we denote $\Cat{n}{x+1}$.  It is worth noting that the more well-known {Touchard polynomials} are given by the functions $x^n \Bell{n}{1/x}$ and that  the functions $x^n \Cat{n}{1/x} = x\Cat{n}{x}$ are sometimes (e.g., in \cite{MansourSun}) also called Narayana polynomials.  
%Generally speaking, in this work the relationship between $f_n(x)$ and $x^n f_n(1/x)$, when $f$ belongs to alphabet $\{ \cB, \mathcal{N}, \dots\}$, is that the first polynomial counts set partitions whose arcs are labeled while the second polynomial counts set partitions whose blocks are labeled.

\end{itemize}
\end{remarks}

%By a \emph{distinguished block} in a partition of $[n]$, we mean a pair $(B,\Lambda)$ where $B$ is a block of $\Lambda \vdash[n]$.

The following theorem shows that 
$\Bell{n}{x}$   counts the  $\Lin(n+1,\BB)$-orbits in $\sP(n+1,\AA,\BB)$ while $\Mo{n}{x}$ counts the $\Lin(n+1,\BB)$-orbits in $\NC(n+1,\AA,\BB)$.

\begin{theorem}\label{main-thm} For $n\geq 1$,
the correspondence $\Lambda \mapsto \{ \alpha + \shift(\Lambda) : \alpha \in \Lin(n,\BB)\}$ 
  is a bijection 
\[ \ba \sP(n-1,\AA)&\to
\left\{  \Lin(n,\BB)
\text{-orbits in }\sP(n,\AA,\BB)\right\},
  \\
 \left \{ \text{Poor elements of $\NC(n-1,\AA)$}\right\} &\to  
\left\{ \Lin(n,\BB)\text{-orbits in } \NC(n,\AA,\BB)\right\}.
 \ea
  \] 
Furthermore,  the cardinality of the $\Lin(n,\BB)$-orbit of $\shift(\Lambda)$ is $|\BB|^s$, where $s$ is the number of singleton blocks of $\Lambda \in \sP(n-1,\AA)$.
%  
%  Consequently, if $|\AA| = x+1$ and $|\BB| = y+1$ and $\Lambda \in \sP(n-1,\AA)$ has $s$ singleton blocks, then
%  \begin{enumerate}
%  \item[(1)] There are $\Bell{n-1}{x}$   distinct $\Lin(n,\BB)$-orbits in $\sP(n,\AA,\BB)$;
%  \item[(2)] There are $\Mo{n-1}{x}$ distinct $\Lin(n,\BB)$-orbits in $\NC(n,\AA,\BB)$;
%  
%  \item[(3)] The $\Lin(n,\BB)$-orbit of $\shift(\Lambda)$ has cardinality $(y+1)^s$.%, where $s$ is the number of singleton blocks of $\Lambda \in \sP(n-1,\AA)$.
%  \end{enumerate}
\end{theorem}

\begin{proof}
It is clear from our discussion of the action $+$ that each $\Lin(n,\BB)$-orbit in $\sP(n,\AA,\BB)$ contains a unique 2-regular element %(i.e., whose matrix has no nonzero entries on the superdiagonal $\{(j,j+1) : j \in [n-1]\}$)
 and which is consequently of the form $\shift(\Gamma)$ for a unique $\Gamma \in \sP(n-1,\AA)$.  
This shows that the first map is a bijection; the second map is also because $\shift(\Gamma)$ is noncrossing if and only if $\Gamma$ is poor and noncrossing, as a consequence of Lemma \ref{shift-bij}.

If $\Lambda = \shift(\Gamma)$, then $\Arc(\alpha) \cap \Arc(\Lambda) = \varnothing$ for all $\alpha \in \Lin(n,\BB)$, and one sees directly from Definition \ref{+def}
that 
the $\Lin(n,\BB)$-orbit of $\Lambda$ has size $|\BB|^s$ where $s$ is the number of superdiagonal hooks $\{(i,j+1) : i < j \} \cup \{(j,k): j<k\}$ for $j \in [n-1]$ which contain no nonzero entries in the matrix of $\Lambda$.  Consulting the definition of $\shift$, we find that the hook containing $(j,j+1)$ belongs to this set if and only if $\{j\}$ is a singleton block of $\Gamma$.
\end{proof}

The following theorem uses the previous result to derive two equivalent formulas for each of the polynomials $\Bell{n}{x,y}$ and $\Cat{n}{x,y}$.

\begin{theorem}\label{identities}
If $n$ is a nonnegative integer then the following identities hold:
\begin{enumerate}
\item[(1)]
$\Bell{n+1}{x,y} =\ds \sum_{k=0}^{n} \binom{n}{k} \Bell{k}{x} y^{n-k} = %\sum_{i,j} \stirl{n-j}{n-i-j}\binom{n}{j} x^i y^j =
 \sum_{k=0}^n \binom{n}{k}\Fe{k}{x} (y+1)^{n-k}$.

\item[(2)]
$\Cat{n+1}{x,y} = \ds\sum_{k=0}^n \binom{n}{k} \Mo{k}{x} y^{n-k}
=
\sum_{k=0}^{\lfloor n/2\rfloor} \cC_k \binom{n}{2k} x^k  (y+1)^{n-2k}$.
\end{enumerate}
\end{theorem}

\begin{proof}
In each part, the terms in first sum counts partitions  whose matrices have  $n-k$ nonzero entries on the superdiagonal $\{ (i,i+1) : i\in [n]\}$, while the terms in the second sum count partitions whose $\Lin(n+1,\BB)$-orbits have the same fixed size.  Thus the sums in each part are necessarily equal to each other and to $\Bell{n+1}{x,y}$ in (1) and $\Cat{n+1}{x,y}$ in (2).

In detail, $\binom{n}{k} \Bell{k}{x} y^{n-k}$ is 
the number of  elements of $\sP(n+1,\AA,\BB)$ whose matrices have  $n-k$ nonzero entries on the superdiagonal since there are $\binom{n}{k} y^{n-k}$ choices for the entries and their positions, and since deleting the rows and columns containing these positions produces  the matrix of $\shift(\Lambda)$ for an arbitrary $\Lambda \in \sP(k,\AA)$.  
On the other hand, $\binom{n}{k}\Fe{k}{x} (y+1)^{n-k}$ is the number of elements of $\sP(n+1,\AA,\BB)$ whose $\Lin(n+1,\BB)$-orbits have size $(y+1)^{n-k}$ by Theorem \ref{main-thm}, since there are $\binom{n}{k} \Fe{k}{x}$ distinct $\AA$-labeled set partitions of $[n]$ with $n-k$ singleton blocks. 

Likewise, $\binom{n}{k} \Mo{k}{x} y^{n-k}$ is 
the number of  elements of $\NC(n+1,\AA,\BB)$ whose matrices have  $n-k$ nonzero entries on the superdiagonal since deleting the rows and columns containing these positions produces  the matrix of $\shift(\Lambda)$ for an arbitrary poor $\Lambda \in \NC(k,\AA)$.  
By Theorem \ref{main-thm}, the number of elements of $\NC(n+1,\AA,\BB)$ whose $\Lin(n+1,\BB)$-orbits have size $(y+1)^{n-k}$ is equal to the number of poor elements of $\NC(n,\AA)$ with $n-k$ singleton blocks,
which is the product of $\binom{n}{n-k}$ with the number of partitions in $\NC(k,\AA)$ whose blocks all have size two.  The latter number is clearly 0 if $k$ is odd, and is equal to $x^{k/2}$ times the leading coefficient of $\Mo{k}{x}$ if $k$ is even. 
\end{proof}

\begin{remarks} Both parts of the theorem deserve a few comments.
\begin{enumerate} 
\item[(i)] Setting $y=0$ in the first part shows that $\Bell{n}{x}$ is the binomial transform of $\Fe{n}{x}$; i.e.,  $\Bell{n}{x} = \sum_{k} \binom{n}{k} \Fe{k}{x}$ (see \cite[Lemma 5.2]{M_normal} for a combinatorial proof). % It follows from (\ref{fe-rec}) that $\Bell{n}{1} = \Fe{n}{1} + \Fe{n+1}{1}$, an identity discussed in \cite[Section 3.5]{Bernhart}. % Extracting  the coefficients of $x^k$ in this equation provides the well-known identity $\stirl{n}{k} = \sum_i \binom{n}{i} \stst{n-i}{k-i}$. 
 Setting $x=y$ in (1) yields 
$\Bell{n+1}{x} =\sum_{k=0}^n \binom{n}{k} \Bell{k}{x} x^{n-k},$ an identity noted several places previously \cite{KerberFollow,Rogers}, which is equivalent to the standard recurrence formula for the Touchard polynomials.

\item[(ii)] We may rewrite the second  part  as the equation
\[ \sum_{i=0}^{\lfloor n/2\rfloor} \sum_{j=0}^{n-2i} \frac{1}{n+1} \binom{n+1}{i+1} \binom{n-i}{j} \binom{n-i-j}{i} x^{i} y^j
=
\sum_{k=0}^{\lfloor n/2\rfloor} \cC_k \binom{n}{2k} x^k (y+1)^{n-2k}.\]
 Chen, Deutsch, and Elizalde give a  combinatorial proof of this identity using a correspondence between plane trees and 2-Motzkin paths \cite[Theorem 9]{CDE}.
 Setting $x=y$ in part (2) produces Coker's identity (\ref{intro1}) mentioned in the introduction
% \be\label{coker-id}\Cat{n+1}{x} = \sum_{k=0}^{n} \frac{1}{n+1}\binom{n+1}{k} \binom{n+1}{k+1}x^{k} = \sum_{k=0}^{\lfloor n/2\rfloor} \cC_k \binom{n}{2k} x^k (x+1)^{n-2k},\ee which  Coker \cite[Eq.\ (4.4)]{Coker_Enum} first derived  using generating functions and which Chen, Yan, and Yang \cite{CYY} later proved combinatorially using a weighted version of a bijection between Dyck paths and 2-Motzkin paths. 
and taking $x=y=1$ recovers Touchard's classical identity $\cC_{n+1} = \sum_k \cC_k \binom{n}{2k} 2^{n-2k}$.
 \end{enumerate}
\end{remarks}

%
%Theorem \ref{main-thm} shows that the $\Lin(n+1,\BB)$-invariant elements of $\sP(n+1,\AA,\BB)$ and $\NC(n+1,\AA,\BB)$ correspond via $\shift$ to the elements of $\sP(n,\AA)$  with no singleton blocks and   to the   elements of $\NC(n,\AA)$ whose blocks all have size two.  The polynomial $\Fe{n}{x}$ therefore counts the invariant elements of $\sP(n+1,\AA,\BB)$, while the number of invariant elements of $\NC(n+1,\AA,\BB)$ is $x^{n/2} \cC_{n/2}$ if $n$ is even and zero otherwise. 
%One consequence of this fact is the following  corollary.
%
%\begin{corollary}\label{+cor} The involution $\Lambda \mapsto \Lambda^+$ has
%$\Fe{n}{1}$, $ \cC_{n}$, and 0 fixed points in $\sP(n+1)$, $\NC(2n+1)$, and $\NC(2n+2)$, respectively, for $n\geq 0$.
%\end{corollary}
%
%\begin{proof}
%This follows from the preceding discussion given the fact, easily checked from the definitions, that an $\FF_2$-labeled partition $\Lambda \vdash[n]$ is $\Lin(n,\FF_2)$-invariant if and only if $\Lambda^+ =\Lambda$.
%\end{proof}

%The values of these polynomials with $x=1$ are the respective numbers of set partitions in $\sP(n+1)$ and $\NC(n+1)$ which are invariant under the involution $\Lambda \mapsto \Lambda^+$.
%

As one consequence of the theorem, we employ the inclusion-exclusion principle to compute  alternate formulas counting the invariant elements in $\sP(n,\AA,\BB)$ and $\NC(n,\AA,\BB)$.

\begin{corollary} \label{main-cor}
%The correspondence $\Lambda \mapsto \shift(\Lambda)$ is a bijection 
%\[ \ba\left \{\text{Feasible elements of $\sP(n-1,\AA)$}\right\}&\to
%\left\{\text{$\Lin(n,\BB)$-invariant elements of $\sP(n,\AA,\BB)$}\right\},
%%\left  \{ \text{Invariant supercharacters of $\UT_{n+1}(\FF_q)$} \right\},
%  \\
% \left \{ \text{Feasible poor elements of $\NC(n-1,\AA)$}\right\} &\to
%\left\{\text{$\Lin(n,\BB)$-invariant elements of $\NC(n,\AA,\BB)$}\right\}.
%% \left \{\text{Invariant irreducible supercharacters of $\UT_{n+1}(\FF_q)$}\right \}.
%  \ea
%\]   Consequently, %if $|\AA| = x+1$ and 
For each integer $n\geq 0$, the following identities hold:
\begin{enumerate}
\item[(1)] 
$\ds \sum_{k=0}^n (-1)^{n-k} \binom{n}{k}  (y+1)^{n-k} \Bell{k+1}{x,y} =\Fe{n}{x}$.

\item[(2)] $\ds\sum_{k=0}^n (-1)^{n-k} \binom{n}{k} (y+1)^{n-k}  \Cat{k+1}{x,y}= \begin{cases} x^{n/2} \cC_{n/2},&\text{if $n$ is even,}
\\
0, &\text{if $n$ is odd.}
\end{cases}$
\end{enumerate}

%
%\begin{enumerate}
%\item[(1)] The identity 
%\[\sum_{k=0}^n (-1)^{n-k} \binom{n}{k}  (y+1)^{n-k} \Bell{k+1}{x,y} =\Fe{n}{x}\]
%holds and gives the number of $\Lin(n+1,\BB)$-invariant elements of $\sP(n+1,\AA,\BB)$;
%
%\item[(2)] The identity 
%\[\sum_{k=0}^n (-1)^{n-k} \binom{n}{k} (y+1)^{n-k}  \Cat{k+1}{x,y}= \begin{cases} 0, &\text{if $n$ is odd} \\
%x^{n/2} \cC_{n/2},&\text{if $n$ is even}\end{cases}\]
%holds and gives the number of $\Lin(n+1,\BB)$-invariant elements of $\NC(n+1,\AA,\BB)$.
%
%%\item[(3)] $\NC(2n+1,\AA,\BB)$ has $x^n \cC_n$ distinct invariant elements.
%\end{enumerate}
  \end{corollary}

\begin{remark}
Setting $x=y$ in  part (2) recovers one of the three identities given by Mansour and Sun in \cite[Theorem 1.1]{MansourSun}. This equation with $x=y=1$ appeared  earlier as \cite[Proposition 2.2]{BSS} and has been studied in a number of places;  Chen, Wang, and Zhao provide a nice bibliography in the discussion preceding \cite[Theorem 2.4]{CWZ}.
\end{remark}

\begin{proof}
Given a subset $S \subset [n]$, let
 $\cX_S$ and $\cY_S$ denote the unions of the $\Lin(n,\BB)$-orbits
 of $\shift(\Lambda)$ for all partitions $\Lambda$ in $\sP(n,\AA)$ and $\NC(n,\AA)$, respectively,
which contain the singleton $\{i\}$ as a block for each $i \in S$. It is straightforward to compute from Theorems \ref{main-thm} and \ref{identities} that
\[
|\cX_S| = (y+1)^{|S|} \Bell{n-|S|+1}{x,y}\qquad\text{and}\qquad
|\cY_S| = (y+1)^{|S|} \Cat{n-|S|+1}{x,y}.\]
The inclusion-exclusion principle now affords our result since
by Theorem \ref{main-thm} the sets of $\Lin(n,\BB)$-invariants in $\sP(n+1,\AA,\BB)$ and $\NC(n+1,\AA,\BB)$ are the respective complements of $\bigcup_{i \in [n]}\cX_{\{ i\}}$ and 
$\bigcup_{i \in [n]}\cY_{\{ i\}}$.
\end{proof}

\section{A short digression on nonnesting partitions}
\label{nonnesting}

\def\nc{\text{\emph{uncross}}}
\def\ncb{\nc_{\mathrm{B}}}
\def\nn{{unnest}}

In the next section we introduce type B and D analogues for the sets $\sP(n,\AA)$ and $\NC(n,\AA)$.  Before studying these new families of set partitions, 
it will be useful to prepare the way with some requisite notation.

To this end, we recall that a set partition $\Lambda$ is \emph{nonnesting} if no two arcs $(i,l),(j,k) \in \Arc(\Lambda)$ have   
 $i < j < k < l$. Visually, this means that no arc ``nests'' beneath another  in $\Lambda$'s standard representation.  Let $[\pm n] = \{\pm 1,\pm 2,\dots, \pm n\}$ and write $-\Lambda$ for the set partition whose blocks are $-B$ for $B \in \Lambda$.
Now define 
\[\ba  \NN(n) &\omdef=\text{the set of nonnesting  partitions of $[n]$}, \\
\NNB(n) &\omdef=\text{the set of nonnesting  partitions $\Lambda$ of $[\pm n]$ with $\Lambda =- \Lambda$.}
\ea\]
%where $\dag $ is the involution on $\sP(n)$ induced by the map $i \mapsto n+1-i$.  
%
The elements of $\NN(n+1)$ and $\NNB(n)$ can be viewed as the type $A_n$ and $B_n$ instances of a more general object called a ``nonnesting partition'' with many interesting properties.   We direct the reader to \cite{memoir,Athanasiadis} for a more detailed exposition, as here we shall only discuss  a few basic facts.

To begin, there is a close relationship between $\NN(n)$ and $\NC(n)$: the number of nonnesting and noncrossing set partitions of $[n]$ are equidistributed by type \cite[Theorem 3.1]{Athanasiadis}, where the \emph{type} of a set partition is the  partition of the number $n$ whose parts are the sizes of the set partition's blocks.  %The number of nonnesting and noncrossing set partitions of $[n]$ are equidistributed by type; see \cite[Theorem 3.1]{Athanasiadis} for a bijective proof.
The following simple bijection 
$\nc : \NN(n) \to \NC(n)$ is not type-preserving but will be of some use later.  We define this by the algorithm below (see also \cite[Section 5.1]{memoir}):
%
%
%Given $\Lambda \in \NN(n)$, let $\nc(\Lambda) \in \NC(n)$ be the noncrossing set partition of $[n]$ formed by the following algorithm:
\begin{enumerate}
\item[1.] Given any set partition $\Lambda$, let $A = \Arc(\Lambda)$.
\item[2.] While $A$ has at least one pair of crossing arcs: choose two arcs $(i,k),(j,l) \in A$ with $i<j<k<l$ and replace $A$ with $\(A - \{(i,k),(j,l)\} \)\cup \{ (i,l),(j,k) \}$ .
%\begin{enumerate}
%\item[a.] Choose any two arcs $(i,k),(j,l) \in A$ with $i<j<k<l$.
%\item[b.] Replace $A$ with $\(A - \{(i,k),(j,l)\} \)\cup \{ (i,l),(j,k) \}$.
%\end{enumerate}
\item[3.] Define $\nc(\Lambda)$ as the noncrossing set partition of $[n]$ with arc set $A$. 
\end{enumerate}
Note that this algorithm makes sense for any partition of a finite set of integers.
The procedure  locally converts each crossing to a nesting in the standard representation of $\Lambda$; i.e.,
\[ 
\xy<0.0cm,.25cm> \xymatrix@R=0.5cm@C=.5cm{
*{} & 
*{}  &
*{}  
\\
*{} \ar @/^.6pc/ @{-} [rru]   & 
*{}  &
*{} \ar @/_.6pc/ @{-} [llu] 
}\endxy
\qquad\text{is locally replaced with}
\qquad
\xy<0.0cm,0.25cm> \xymatrix@R=0.5cm@C=.5cm{
*{} \ar @/^.2pc/ @{-} [rr] & 
*{}  &
*{}  
\\
*{}   & 
*{}  &
*{} \ar @/_1.1pc/ @{-} [ll] 
}\endxy
\]
so that, e.g., $\nc\( 14|25|36\) = 16|25|34$ and $\nc(\Lambda) = \Gamma$ in (\ref{std-ex}).
This observation makes clear that the algorithm's output has 
 no dependence on the order in which the pairs of crossing arcs are chosen in the second step. Thus $\nc$ gives a well-defined map $\sP(n) \to \NC(n)$, the  important properties of which are the following:
 \begin{enumerate}
 \item[$\bullet$] $\nc(\Lambda)$ has the same number of blocks as $\Lambda$.
 \item[$\bullet$] $\nc$ defines a bijection from  $\NN(n) \to \NC(n)$.
 \item[$\bullet$] $\nc$ defines a bijection     $\NNB(n)\to \{\text{Noncrossing  $\Lambda\vdash [\pm n]$ with $\Lambda = -\Lambda$}\}$.
 \end{enumerate}
The first property is clear since $\nc(\Lambda)$ has the same number of arcs as $\Lambda$;  the second property is well-known; and the third follows from the second since $\nc(-\Lambda) = -\nc(\Lambda)$.

We see from this discussion that $\NN(n)$ has $N(n,k)$ elements with $k$ blocks and $\cC_n$ elements in total. The following observation lists analogous statistics for $\NNB(n)$. 

\begin{lemma} \label{obs1}
$\NNB(n)$ has $\binom{n}{k}^2$ elements with $2k$ or $2k+1$ blocks, and $\binom{2n}{n}$ elements in total.
\end{lemma}

\begin{proof}
The involution on set partitions of $[\pm n]$ induced by the map
\[i \mapsto \begin{cases} i-n-1, &\text{if }i \in [n] \\ i+n+1,&\text{if } -i \in [n]\end{cases}\]
gives a bijection from $\NNB(n)$ to what
Athanasiadis \cite{Athanasiadis} calls the $B_n$-partitions associated to the nonnesting partitions of type $B_n$.
Hence the lemma is simply a restatement of \cite[Corollary 5.8]{Athanasiadis}.
\end{proof}

Recall that a \emph{Dyck path} with $2n$ steps is a lattice path beginning at $(0,0)$ and ending at $(n,0)$ which uses only the steps $(1,1)$ and $(1,-1)$ and which never travels below the $y$-axis.   It is well-known that the set $\cD_n$ of Dyck paths with $2n$ steps  has cardinality $|\cD_n| = |\NC(n)| = |\NN(n)| = \cC_n$, and there is an especially simple bijection $\NN(n) \to \cD_n$.  Namely, we associate to a nonnesting set partition $\Lambda \in \NN(n)$ the unique path in $\cD_n$ whose valleys (the points which simultaneously end a downstep $(1,-1)$ and begin an upstep $(1,1)$)
are  the points $(j+i-1,j-i-1)$ for $(i,j ) \in \Arc(\Lambda)$. Intuitively, this is the path tracing the upper boundary of the squares in the matrix of $\Lambda$ which are south or west of nonzero entries.
%; for example,
%\[ 
%\xy<0.0cm,-0.0cm> \xymatrix@R=-0.0cm@C=.3cm{
%*{\bullet} \ar @/^.7pc/ @{-} [rr]   & 
%*{\bullet}  &
%*{\bullet}  &
%*{\bullet}
%\\
%1   & 
%2 &
%3  &
%4
%}\endxy
%\quad\text{has matrix}\quad 
%\barr{|c|c|c|c|c|c|c|} \hline
%0 &0 &1 &0 \\ \hline
% & 0 &0 &0  \\ \hline
% & & 0& 0 \\ \hline
% & & & 0 \\ \hline
% \earr
%\quad\text{and corresponds to}\quad
%\xy<0.0cm,0.6cm> \xymatrix@R=0.2cm@C=.2cm{
%*{}  & 
%*{} &
%*{}  &
%*{}  &
%*{} &
%*{\bullet}\ar @{-} [dddrrr]   &
%*{}
%\\
%*{}  & 
%*{} &
%*{\bullet}  \ar @{-} [dr] &
%*{}  &
%*{\bullet}&
%*{}  &
%*{\bullet}
%\\
%*{}  & 
%*{\bullet} &
%*{}  &
%*{\bullet} \ar @{-} [uurr]  &
%*{} &
%*{}  &
%*{}  &
%*{\bullet}
%\\
%*{\bullet} \ar @{-} [uurr]  & 
%*{} &
%*{}  &
%*{} &
%*{} &
%*{}  &
%*{}  &
%*{}  &
%*{\bullet}
%}\endxy
% \]

Call a Dyck path with $2n$ steps \emph{symmetric} if the path is symmetric about the vertical line $x = n$.  The order-preserving bijection $[\pm n] \to [2n]$ induces an inclusion of  $\NNB(n)$ in $\NN(2n)$, and it clear that with respect to this identification,
the bijection $\NN(2n) \to \cD_{2n}$ just mentioned restricts to a bijection from $\NNB(n)$ to the set of symmetric Dyck paths with $4n$ steps.
Hence:
\begin{lemma}\label{obs2} There are $ \binom{2n}{n}$ symmetric Dyck paths with $4n$ steps.
\end{lemma}

\section{Analogues in other types}
\label{b-section}

\def\so{\mathfrak{so}}

We are now prepared to discuss two analogues for our ``classical'' notion of a labeled set partition.  In detail,
given a nonnegative integer and an additive abelian group, we 
define $\sPB(n,\AA)$ (respectively, $\sPD(n,\AA)$) as the set of $\AA$-labeled set partitions $\Lambda$ of $ \{0\}\cup [\pm n]$ (respectively, $[\pm n]$) with the property that 
\be\label{b-def}(i,j) \in \Arc(\Lambda)\quad\text{if and only if}\quad (-j,-i) \in \Arc(\Lambda)\text{ and }\Lambda_{ij} + \Lambda_{-j,-i}=0.\ee  
We write $\sPB(n) \omdef = \sPB(n,\FF_2)$ and $\sPD(n) = \sPD(n,\FF_2)$
for the corresponding sets of unlabeled set partitions. 

The condition (\ref{b-def}) implies that $\Lambda$ has no arcs of the form $(-i,i)$ and hence that $|\Arc(\Lambda)|$ is even, and   that if $B$ is a block of $\Lambda$ then $-B$ is also a block. If $\Lambda \in \sPB(n,\AA)$ then exactly one block $B \in \Lambda$ contains zero and has $B=-B$, while if $\Lambda \in \sPD(n,\AA)$ then every block $B \in \Lambda$ has $B\neq -B$.
For example, the elements of $\sPB(2,\AA)$ are
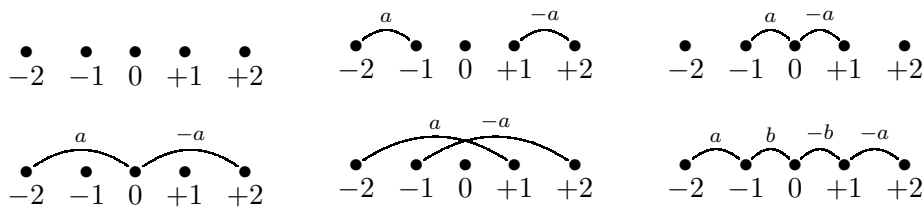
\[\ba 
 \xy<0.0cm,-0.0cm> \xymatrix@R=-0.0cm@C=.1cm{
*{\bullet}  &
*{\bullet} &
*{\bullet}   &
*{\bullet}&
*{\bullet} \\
-2   & 
-1 &
0  &
+1 &
+2 
}\endxy
\qquad
 \xy<0.0cm,-0.0cm> \xymatrix@R=-0.0cm@C=.1cm{
*{\bullet} \ar @/^.5pc/ @{-} [r]^a &
*{\bullet}  &
*{\bullet} &
*{\bullet}\ar @/^.5pc/ @{-} [r]^{-a}&
*{\bullet} \\
-2   & 
-1 &
0  &
+1 &
+2 
}\endxy
\qquad
 \xy<0.0cm,-0.0cm> \xymatrix@R=-0.0cm@C=.1cm{
*{\bullet}  &
*{\bullet} \ar @/^.5pc/ @{-} [r]^a &
*{\bullet} \ar @/^.5pc/ @{-} [r]^{-a}  &
*{\bullet}&
*{\bullet} \\
-2   & 
-1 &
0  &
+1 &
+2 
}\endxy
\\
 \xy<0.0cm,-0.0cm> \xymatrix@R=-0.0cm@C=.1cm{
*{\bullet} \ar @/^.7pc/ @{-} [rr]^a &
*{\bullet} &
*{\bullet} \ar @/^.7pc/ @{-} [rr]^{-a} &
*{\bullet}&
*{\bullet} \\
-2   & 
-1 &
0  &
+1 &
+2 
}\endxy
\qquad
 \xy<0.0cm,-0.0cm> \xymatrix@R=-0.0cm@C=.1cm{
*{\bullet} \ar @/^.9pc/ @{-} [rrr]^a &
*{\bullet} \ar @/^.9pc/ @{-} [rrr]^{-a} &
*{\bullet} &
*{\bullet}&
*{\bullet} \\
-2   & 
-1 &
0  &
+1 &
+2 
}\endxy
\qquad
 \xy<0.0cm,-0.0cm> \xymatrix@R=-0.0cm@C=.1cm{
*{\bullet} \ar @/^.5pc/ @{-} [r]^a &
*{\bullet} \ar @/^.5pc/ @{-} [r]^b &
*{\bullet} \ar @/^.5pc/ @{-} [r]^{-b}  &
*{\bullet}\ar @/^.5pc/ @{-} [r]^{-a} &
*{\bullet} \\
-2   & 
-1 &
0  &
+1 &
+2 
}\endxy
\ea\] 
and the elements of $\sPD(2,\AA)$ are
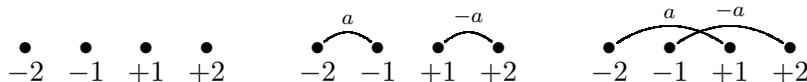
\[
 \xy<0.0cm,-0.0cm> \xymatrix@R=-0.0cm@C=.1cm{
*{\bullet}  &
*{\bullet} &
*{\bullet}&
*{\bullet} \\
-2   & 
-1 &
+1 &
+2 
}\endxy
\qquad
 \xy<0.0cm,-0.0cm> \xymatrix@R=-0.0cm@C=.1cm{
*{\bullet} \ar @/^.5pc/ @{-} [r]^a &
*{\bullet}  &
*{\bullet}\ar @/^.5pc/ @{-} [r]^{-a}&
*{\bullet} \\
-2   & 
-1 &
+1 &
+2 
}\endxy
\qquad
 \xy<0.0cm,-0.0cm> \xymatrix@R=-0.0cm@C=.1cm{
*{\bullet} \ar @/^.7pc/ @{-} [rr]^a  &
*{\bullet}  \ar @/^.7pc/ @{-} [rr]^{-a}&
*{\bullet}  &
*{\bullet} \\
-2   & 
-1 &
+1 &
+2 
}\endxy
\] 
for $a,b \in \AA\setminus\{0\}$.

We have three reasons to suggest these  sets as the type B and D analogues of $\sP(n,\AA)$.  First and possibly foremost, $\sPB(n,\FF_q)$ and $\sPD(n,\FF_q)$ are the natural indexing sets for the supercharacters defined in \cite{bcd1,bcd2,bcd3} of the 
the Sylow $p$-subgroups of the Chevalley groups of  type $B_n$ and $D_n$ over $\FF_q$ (where $p$ is the odd characteristic of $\FF_q$); see Section \ref{last-sect} below for an explicit definition.
Thus, in analogy with our techniques in Section \ref{classical}, we can use the multiplicative action of the linear characters of these groups on their supercharacters to define a combinatorial action of a subset of $\sPB(n,\AA)$ or $\sPD(n,\AA)$ on itself. From an analysis of the orbits of  this action, we may then attempt to derive identities in the style of Theorem \ref{identities}.  

The papers \cite{bcd1,bcd2,bcd3} also define a set of supercharacters for the Sylow $p$-subgroups of the finite Chevalley groups of type $C_n$, which should motivate the definition of a third family of labeled set partitions $\sPC(n,\AA)$.  We omit this family from the present work because its investigation fits less naturally into our exposition and seems  not to yield any really new identities.

A second explanation for our notation comes from the following observation.
The order-preserving bijections $\{ 0\} \cup [\pm n] \to [2n+1]$ and $[\pm n] \to [2n]$ induce  inclusions $\sPB(n,\AA) \hookrightarrow \sP(2n+1,\AA)$ and $\sPD(n,\AA)\hookrightarrow \sP(2n,\AA)$, and we define the \emph{matrix} of $\Lambda$ in $\sPB(n,\AA)$ or $\sPD(n,\AA)$ to be the matrix of the corresponding set partition in $\sP(2n+1,\AA)$ or $\sP(2n,\AA)$. 
If $\CC$ is the field of complex numbers and 
\[ \so'(n,\CC) \omdef= \{ X \in \sl(n,\CC) : X + X^\dag = 0\}, \quad\text{where }(X^\dag)_{ij} \omdef= X_{n+1-j,n+1-i},\]
 then the map assigning a set partition to its matrix gives a bijection from $\sP(n+1,\CC)$, $\sPB(n,\CC)$, and $\sPD(n,\CC)$ to the sets of strictly upper triangular matrices with at most one nonzero entry in each row and column in $\sl(n+1,\CC)$,  $\so'(2n+1,\CC)$, and $\so'(2n,\CC)$,  which we may regard as the complex simple Lie algebras of types $A_n$, $B_n$, and $D_n$.

Finally, we mention that the unlabeled partitions $\sP(n+1)$ and $\sPB(n)$ are naturally identified with the intersection lattice of the Coxeter hyperplane arrangements of type $A_n$ and $B_n$ (see \cite{Reiner_NC}). The set $\sPD(n)$ is not similarly related to the intersection lattice of the type $D_n$ Coxeter hyperplane arrangement, however.

The sets of noncrossing elements of  $\sPB(n,\AA)$ and $\sPD(n+1,\AA)$ are both in  bijection with $\NC(n,\AA)$, so it will be fruitful to introduce a different kind of ``noncrossing'' partition to investigate.  To this end,
for $\X \in \{ \mathrm{B}, \mathrm{D}\}$, 
we let
 $\tNCX(n,\AA)$ denote the  subset of $\sPX(n,\AA)$ consisting of $\AA$-labeled set partitions $\Lambda$ with the following ``noncrossing'' property:
\[\text{If there are }(i,k),(j,l) \in \Arc(\Lambda)\text{ such that }i<j<k<l\text{ then } (i,k) = (-l,-j).\]
As usual, we write $\tNCX(n) \omdef = \tNCX(n,\FF_2)$
to indicate the corresponding sets of unlabeled set partitions. 
This set generalizes $\NC(n,\AA)$ in the following sense: one may define $\NC(n,\AA)$ as the subset of $\Lambda \in \sP(n,\AA)$ such that if $(i,k), (j,l) \in \Arc(\Lambda)$ and $\Gamma \in \sP(n,\AA)$ then $\{ (i,j),(j,k),(k,l)\} \not\subset \Arc(\Gamma)$.  The same definition with $\sP(n,\AA)$ replaced by $\sPX(n,\AA)$ gives $\tNCX(n,\AA)$.  Alternatively, Theorem \ref{last-thm} below provides a representation theoretic characterization of $\tNCX(n,\AA)$ in terms of its associated set of supercharacters.

As with $\sP(n)$, the sets $\sPB(n)$ and $\sPD(n)$ are partially ordered by refinement, and graded by the rank functions
\[ \rank(\Lambda) \omdef=\begin{cases}  n-(|\Lambda|-1)/2,&\text{for $\Lambda \in \sPB(n)$}, \\
n - |\Lambda|/2,&\text{for $\Lambda \in \sPD(n)$}.\end{cases}\]
Both $\sPB(n)$ and $\sPD(n)$ are meet semilattices since any collection of elements $\{\Lambda_i\}$ has a greatest lower bound given by  the partition whose blocks are the nonempty intersections of the form $ \bigcap_i B_i$ where each $B_i \in \Lambda_i$.
However, of the two, only $\sPB(n)$ possesses a greatest element and is therefore a lattice.
The meet of any collection of elements in $\tNCX(n)$ also lies in  $\tNCX(n)$ for $\X \in \{\mathrm{B},\mathrm{D}\}$, and it follows that $\tNCB(n)$ is likewise a graded lattice (with height $n$) while $\tNCD(n)$ is only a graded meet semilattice (with height $n-1$).

 \begin{remark} 
 To any Coxeter system $(W,S)$ there corresponds a lattice of noncrossing partitions, defined as the interval between the identity and any fixed Coxeter element in the  absolute order of $W$.  There is a large body of literature on this subject; see  \cite{memoir} for a useful survey.
  The noncrossing partition lattice of the Coxeter system of type $A_{n-1}$ coincides with  $\NC(n)$, and the lattices 
  of types $B_n$ and $D_n$, which we might as well denote by $\NCB(n)$ and $\NCD(n)$,  may be realized combinatorially as subposets of $\sPB(n)$.  However,  $\NCB(n)$ and $\NCD(n)$ are neither obviously related to $\tNCB(n)$ and $\tNCD(n)$ (though there are connections between them), 
  nor preserved by the group action defined below, so the somewhat more obscure sets $\tNCB(n)$ and $\tNCD(n)$ are better suited to our  purposes.
   \end{remark}

Write $\LinX(n,\AA)$ to denote the type X analogue of $\Lin(n,\AA)$: viz., the set of labeled partitions $\Lambda \in \sPX(n,\AA)$ whose blocks consist of consecutive integers or, equivalently, which have $\Arc(\Lambda) = \Cov(\Lambda)$.  
We define $\alpha+\Lambda$ for $\alpha \in \LinX(n,\AA)$ and $\Lambda \in \sPX(n,\AA)$ exactly as in Definition \ref{+def}, only with 
 $\Lin(n,\AA)$ and $\sP(n,\AA)$ replaced  by $\LinX(n,\AA)$ and $\sPX(n,\AA)$.    For example, if
 \[ \alpha
 =
 \xy<0.0cm,-0.0cm> \xymatrix@R=-0.0cm@C=.1cm{
*{\bullet} \ar @/^.5pc/ @{-} [r]^{a} &
*{\bullet} \ar @/^.5pc/ @{-} [r]^{b} &
*{\bullet} \ar @/^.5pc/ @{-} [r]^{c} &
*{\bullet} \ar @/^.5pc/ @{-} [r]^{-c}  &
*{\bullet}\ar @/^.5pc/ @{-} [r]^{-b} &
*{\bullet} \ar @/^.5pc/ @{-} [r]^{-a} &
*{\bullet} \\
-3 &
-2   & 
-1 &
0  &
+1 &
+2 &
+3
}\endxy
\qquad\text{and}\qquad
\Lambda= \xy<0.0cm,-0.0cm> \xymatrix@R=-0.0cm@C=.1cm{
*{\bullet} \ar @/^.9pc/ @{-} [rrrr]^{t} &
*{\bullet} &
*{\bullet} \ar @/^.9pc/ @{-} [rrrr]^{-t} &
*{\bullet}  &
*{\bullet} &
*{\bullet} &
*{\bullet} \\
-3   & 
-2 &
-1  &
0 &
+1 &
+2 &
+3 
}\endxy
\]
for $a,b,c,d,e,t \in \AA\setminus\{0\}$ then we have
\[ \alpha +\Lambda=
 \xy<0.0cm,-0.0cm> \xymatrix@R=-0.0cm@C=.1cm{
*{\bullet} \ar @/^1.5pc/ @{-} [rrrr]^{t} &
*{\bullet} \ar @/^.5pc/ @{-} [r]^b  &
*{\bullet} \ar @/^1.5pc/ @{-} [rrrr]^{-t} &
*{\bullet}  &
*{\bullet}\ar @/^.5pc/ @{-} [r]^{-b} &
*{\bullet} &
*{\bullet} \\
-3   & 
-2 &
-1  &
0 &
+1 &
+2 &
+3 
}\endxy
\]
Note that we may alternately characterize $\alpha+\Lambda$ as in Observation \ref{+obs}, or in terms of the matrices of $\alpha \in \LinX(n,\AA)$ and $\Lambda \in \sPX(n,\AA)$ as in 
Observation \ref{+def2}. % then holds in the present context, provided we make the same substitution.
As before, the operation $+$ makes $\LinX(n,\AA)$ into an abelian group (isomorphic to $\AA^n$ if $\X = \mathrm{B}$ and to $\AA^{n-1}$ if $\X = \mathrm{D}$) acting on $\sPX(n,\AA)$, and it is evident from our definitions that this action 
preserves
 $\tNCX(n,\AA)$. %, an important justification for our consideration of these partitions in the first place.

Recall the definition in Section \ref{action-sec} of $\Lambda^+$ for an arbitrary set partition $\Lambda$.  
If we view unlabeled set partitions as $\FF_2$-labeled, then this definition amounts to the formulas 
\[\ba
 \Lambda^+& = \{ \{-n,\dots,-1,0,1,\dots,n\}\} + \Lambda,&&\quad\text{for $\Lambda \in \sPB(n)$},
 \\
  \Lambda^+& = \{ \{-n,\dots,-1\},\{1,\dots,n\}\} + \Lambda,&&\quad\text{for $\Lambda \in \sPD(n)$}.
  \ea\]
%for $\Lambda \in \sPD(n)$ one checks that $\Lambda^+ = \{ \{ -n,\dots,-1\},\{1,\dots,n\}\} + \Lambda$.
  Hence $\Lambda \mapsto \Lambda^+$ defines an involution of both $\sPX(n)$ and $\tNCX(n)$ (as sets, not lattices) for $\X \in \{\mathrm{B},\mathrm{D}\}$.
The following analogue of Proposition \ref{alpha-prop}
uses this involution to show that the lattice $\tNCB(n)$ is rank symmetric.

\begin{proposition}
The map $\Lambda \mapsto \Lambda^+$ is  rank inverting on $\tNCB(n)$. That is,
if $\Lambda \in \tNCB(n)$ has $2k+1$ blocks, then $\Lambda^+ \in \tNCB(n)$ has $2(n-k)+1$ blocks. 
\end{proposition}

\begin{remark}
The lattice of type $B_n$ noncrossing partitions is also rank symmetric; in fact, it is self-dual and locally self-dual \cite{memoir,Reiner_NC}. The lattice $\tNCB(n)$ fails to possess any of these stronger properties
when $n\geq 4$.  
\end{remark}

\begin{proof} Observe that the definition of $\nc(\Lambda)$ make sense for $\Lambda \in \sPB(n)$, although  the set partition $\nc(\Lambda) \vdash \{0\}\cup [\pm n]$ may no longer belong to $\sPB(n)$.  
It is straightforward to check that   $\Lambda^+ \in \tNCB(n)$ if $\Lambda \in \tNCB(n)$
and that $\Lambda$ and $\Lambda^+$ have the same crossings, which are always pairs of arcs of the form
 $(- i,  j), (-j,i)$ 
for $i,j \in [n]$.
It follows that 
$\nc(\Lambda)^+ = \nc(\Lambda^+)$.  Since $\nc$ preserves the number of blocks in a set partition and since both $\nc$ and the involution $+$ commute with natural inclusion $\tNCB(n) \hookrightarrow \sP(2n+1)$, our result is a consequence of Proposition \ref{alpha-prop}.
\end{proof}

As a corollary, we  similarly compute the number of blocks in $\Lambda^+$ for $\Lambda \in \tNCD(n)$.

\begin{corollary} 
If $\Lambda \in \tNCD(n)$ then
\[|\Lambda^+| =\begin{cases}  2n+2-|\Lambda|, &\text{if   $-1$ is the greatest element of its block in $\Lambda$}, \\
2n-|\Lambda|,&\text{otherwise}.\end{cases}\]
\end{corollary}

\begin{proof}
Let $\varphi : \tNCD(n) \to \tNCB(n)$ be the injective map which adds the singleton block $\{0\}$ to $\Lambda \in \tNCD(n)$; i.e., $\varphi(\Lambda) = \Lambda \cup \{ \{ 0\} \}$. It is easy to see that if $-1$ is not the greatest element of its block in $\Lambda \in \tNCD(n)$, so that $\Lambda$ has arcs of the form $(-1,i),(-i,1)$ for some $1<i\leq n$, then $\varphi(\Lambda^+) = \varphi(\Lambda)^+$. In this case $|\Lambda^+| = 2n+2-|\Lambda|$ by the previous proposition applied to $\varphi(\Lambda) \in \tNCB(n)$.  On the other hand, if $-1$ is the greatest element of its block in $\Lambda$,
then 
$\varphi(\Lambda)^+$ is formed from $\varphi(\Lambda^+)$ by adding the arcs $(-1,0)$ and $(0,1)$.  Hence $\varphi(\Lambda)^+$ has two fewer blocks than $\varphi(\Lambda^+)$, so we now obtain $|\Lambda^+| = 2n-|\Lambda|$ from the previous proposition applied to $\varphi(\Lambda)$.
\end{proof}

Fix two additive abelian groups $\AA$ and  $\BB$ and 
 define  $\sPX(n,\AA,\BB)$ and $\tNCX(n,\AA,\BB)$ for $\X \in \{\mathrm{B},\mathrm{D}\}$ %entirely analogously to $\sP(n,\AA,\BB)$,
 as the  subsets of $\sPX(n,\AA\oplus \BB)$ and $\tNCX(n,\AA\oplus\BB)$ consisting of labeled set partitions $\Lambda$ satisfying (\ref{ab-cond}).
 Note as in Section \ref{classical} that $+$ defines an action of $\LinX(n,\BB)$ on $\sPX(n,\AA,\BB)$ and $\tNCX(n,\AA,\BB)$. 
Mirroring our previous notation, we define polynomials
\[ 
%\ba
\BellX{n}{x,y}  \omdef= \sum_{\Lambda \in \sPX(n)} x^{\frac{|\Arc(\Lambda)|}{2}} (y/x)^{\frac{|\Cov(\Lambda)|}{2}}
\quad\text{and}\quad
 \CatX{n}{x,y}  \omdef=  \sum_{\Lambda \in \tNCX(n)} x^{\frac{|\Arc(\Lambda)|}{2}} (y/x)^{\frac{|\Cov(\Lambda)|}{2}},
%\ea
\] 
%where we  employ the shorthand 
%\[(xy)^\Lambda \omdef = x^{\frac{1}{2}|\Arc(\Lambda)\setminus \Cov(\Lambda)|} y^{\frac{1}{2}|\Cov(\Lambda)|}.\]
%When  $|\AA| =x+1$ and $ |\BB|=y+1$, the polynomials $\BellX{n}{x,y}$ and $\CatX{n}{x,y}$ give the respective cardinalities of $\sPX(n,\AA,\BB)$ and $\tNCX(n,\AA,\BB)$.
and let
$\BellX{n}{x} \omdef=\BellX{n}{x,x}$ and $\CatX{n}{x} \omdef = \CatX{n}{x,x}.$ 

Clearly $\BellB{n}{x} = \sum_{k=0}^n \stirlb{n}{k} x^{n-k}$ where 
$\stirlb{n}{k}$ denotes the number of set partitions in $\sPB(n)$ with $2k+1$ blocks.  The  numbers $\stirlb{n}{k}$ coincide with the Whitney numbers of the second kind $W_2(n,k)$ studied in \cite{Dowling,Dowling2} and appear as sequence \cite[A039755]{OEIS}.
On the other hand, since the blocks of set partitions in $ \sPD(n)$ comes in pairs $\pm B$,  the preimage of $\Lambda \in \sP(n)$ under the surjection $\sPD(n) \to \sP(n)$ induced by the absolute value map $[\pm n] \to [n]$
contains $2^{n-|\Lambda|} = 2^{|\Arc(\Lambda)|}$ elements, all having $2|\Arc(\Lambda)|$ arcs. Hence \be\label{belld-eq} \BellD{n}{x} = \Bell{n}{2x},\quad\text{for all $n\geq 0$}.\ee
We find from these formulas that $\{\BellB{n}{1}\}_{n=0}^\infty = (1,2,6,24,116,648,4088,\dots)$ gives the sequence of Dowling numbers \cite[A007405]{OEIS}  while $\{\BellD{n}{1}\}_{n=0}^\infty = (1,1,3,11,49,257,1539,\dots)$ gives sequence \cite[A004211]{OEIS}.

 The following result generalizes the recurrence  $\Bell{n}{x} = \sum_k \binom{n}{k} \Bell{k}{x}$ noted in the first remark to Theorem \ref{identities}.  The first identity is due essentially  to Spivey \cite{spivey}, who proved it in the special case $x=1$ 
 (the general proof   is not much different
 from the short combinatorial argument in \cite{spivey}). 
  
\begin{proposition} The following recurrences hold for integers $m,n\geq 0$:
\begin{enumerate}
\item[(1)] $\ds\Bell{m+n}{x} = \sum_{j=0}^m \sum_{k=0}^n  x^{m+n-j-k}  j^{n-k}\binom{n}{k}  \stirl{m}{j} \Bell{k}{x}$.

\item[(2)] $\ds\BellB{m+n}{x} = \sum_{j=0}^m \sum_{k=0}^n x^{m+n-j-k} (2j+1)^{n-k} \binom{n}{k} \stirlb{m}{j} \Bell{k}{2x}$.

\end{enumerate}
\end{proposition}

\begin{proof} We  only prove (2) as the proof of (1) is similar.
Let $\cY = \{ \pm (m+k) : k \in [n]\}$  and suppose $\AA$ has $x+1$ elements.
%Given a partition $\Lambda \in \sPB(m+n)$, let $\Gamma \in \sPB(m)$ be the partition whose blocks are the nonempty differences $B\setminus \cY$ for $B \in \Lambda$; let $\Upsilon$ be the partition whose blocks are the nonempty intersections $B\cap \cY$ for $B \in \Lambda$; let $\cS$ be the union of the blocks of $\Upsilon$; and let $f : \cY\setminus \cS \to \Gamma$ be the map which assigns $i \in \cY\setminus \cS$ to the unique block $f(i)$ of $\Gamma$ with the property that $\{ i\} \cup f(i)$ is a subset of a block of $\Lambda$.   
We may construct the elements of $\sPB(m+n,\AA)$ in the following manner. 
First choose a partition $\Gamma$ in $\sPB(m)$ with $2j+1$ blocks; there are $\stirlb{m}{j}$ choices for this.  Next, select a $2(n-k)$-element subset $S\subset \cY$ with $S=-S$ and distribute
the elements of $S$ among the blocks of $\Gamma$ so that the resulting partition $\Gamma'$ has $\Gamma' = -\Gamma'$; there are $\binom{n}{k}$ choices for $S$ and $(2j+1)^{n-k}$ possible distributions, since once we have chosen the blocks to contain the positive elements of $S$ the blocks containing the negative elements are uniquely determined.
  Now, label the $2(m+n-j-k)$ arcs of $\Gamma'$ by nonzero elements of $\AA$ so that $\Gamma'$ satisfies (\ref{b-def}); there are $x^{m+n-j-k}$ such labelings.   
  Finally, choose an $\AA$-labeled partition of $\cY\setminus S$ satisfying (\ref{b-def}) and concatenate this with $\Gamma'$ to form an element of $\sPB(m+n,\AA)$; there are $\BellD{k}{x}$ choices for  this partition.
Each element of $\sPB(m+n,\AA)$ arises from exactly one such construction, so summing the product $x^{m+n-j-k}(2j+1)^{n-k} \binom{n}{k} \stirlb{m}{j} \Bell{k}{2x}$ over all possible values of $j$ and $k$ yields $\BellB{m+n}{x}$.
\end{proof}

To likewise compute $\CatX{n}{x}$, we
recall the definition in the previous section of $\nc(\Lambda)$ for a set partition $\Lambda$. Modifying this construction slightly, for  $\Lambda \in \sPB(n)$, we define $\ncb(\Lambda)$ to be the set partition of $[\pm n]$ formed by removing zero from its block in $\nc(\Lambda)$.  For example,
\[ \ncb\( \xy<0.0cm,-0.0cm> \xymatrix@R=-0.0cm@C=.3cm{
*{\bullet} \ar @/^1.2pc/ @{-} [rrrr] &
*{\bullet} \ar @/^.7pc/ @{-} [rr] &
*{\bullet} \ar @/^1.2pc/ @{-} [rrrr]  &
*{\bullet} \ar @/^.7pc/ @{-} [rr] &
*{\bullet} &
*{\bullet} &
*{\bullet} \\
-3 &
-2   & 
-1 &
0  &
+ 1 &
+ 2 &
+ 3 
}\endxy\)
=
 \xy<0.0cm,-0.0cm> \xymatrix@R=-0.0cm@C=.3cm{
*{\bullet} \ar @/^1.3pc/ @{-} [rrrrr] &
*{\bullet} \ar @/^.9pc/ @{-} [rrr] &
*{\bullet} \ar @/^.5pc/ @{-} [r]  &
*{\bullet} &
*{\bullet} &
*{\bullet} \\
-3 &
-2   & 
-1 &
+ 1 &
+ 2 &
+ 3 
}\endxy
\]
 We note the following properties of these maps in the present context.
 \begin{lemma}\label{uncross-obs}
The maps 
\[ \ba \ncb : \tNCB(n) &\to \{\text{Noncrossing set partitions $\Lambda\vdash[\pm n]$ with $\Lambda = -\Lambda$} \}
\\
\nc : \tNCD(n) &\to \left\{\ba &\text{Noncrossing set partitions $\Lambda \vdash[\pm n]$ with $\Lambda = -\Lambda$ which} \\ &\text{have an even number of blocks $B$ such that $B=-B$}\ea\right\}
\ea
\]
are bijections. Furthermore, if $\Lambda \in \tNCB(n)$ has $2k+1$ blocks then $\ncb(\Lambda)$ has either $2k$ or $2k+1$ blocks.
\end{lemma}

Note that $\nc(\Lambda)$ has the same number of blocks as $\Lambda$ for any set partition $\Lambda$.

\begin{proof}
 Given $\Lambda \vdash[\pm n]$ with $\Lambda =-\Lambda$, let $\cX = \{ (-i_1,i_1),\dots,(-i_\ell,i_\ell)\}$ be the set of arcs of the form $(-i,i) \in \Arc(\Lambda)$, where $i_1 < \dots < i_\ell$, and define $\cY$ as the set of arcs
\[\cY = \begin{cases}\{ ( -i_{2k},i_{2k-1}), (-i_{2k-1},i_{2k}) : k \in [\tfrac{\ell}{2}]\}, &\text{if $\ell$ is even}, \\
\{ (-i_1,0), (0,i_1) \} \cup \{ ( -i_{2k},i_{2k+1}), (-i_{2k+1},i_{2k}) : k\in [\tfrac{\ell-1}{2}]\},&\text{if $\ell$ is odd}.
\end{cases}\]
Let $\Lambda'$ 
be the  set partition
of $\{0 \} \cup [\pm n]$ with arc set $(\Arc(\Lambda)- \cX)\cup \cY$, and when $\ell$ is even, let $\Lambda''$ be the  set partition
of $[\pm n]$ with arc set $(\Arc(\Lambda)- \cX)\cup \cY$.
  Then $\Lambda \mapsto \Lambda'$ is  the two-sided inverse of the first map in the lemma, while $\Lambda \mapsto \Lambda''$ is the two-sided inverse of the second map.
  The last remark concerning the numbers of blocks follows from the fact that $\ncb(\Lambda)$ partitions a set with one less element than $\Lambda$, yet has either equally many or one fewer arcs than
 $\Lambda$. 
\end{proof}

  The first identity in the following proposition is an immediate consequence of the previous lemma and Lemma \ref{obs1}.  The second part follows from the first, given the fact that $\CatD{n+1}{x} = \CatB{n}{x} + nx\Cat{n}{x}$, an identity which we will prove
 in a more general form as Corollary \ref{hanging}.
    
 \begin{proposition}\label{catb-prop} 
 $\ds\CatB{n}{x} = \sum_{k=0}^n \binom{n}{k}^2 x^{k}$ and $\ds\CatD{n+1}{x} = \sum_{k=0}^{n} \binom{n}{k} \binom{n+1}{k} x^k$
for $n\geq 0$.
% \begin{enumerate}
%% \item[(1)] $\BellB{n}{x} = \sum_{k=0}^n \stirlb{n}{k} x^{n-k}$. %, where $\stirlb{n}{k}$ is the integer defined by the recurrence \[ \stirlb{n+1}{k+1} = \stirlb{n}{k} + (2k+1) \stirlb{n}{k+1}\quad\text{for $n,k\geq 0$, with $\stirlb{0}{i} =\delta_{i}\text{ and } \stirlb{i}{0} = 1$}.\]
%%  \item[(2)] $\BellD{n}{x}  = \Bell{n}{2x}$.
% \item[(1)] $\ds\CatB{n}{x} = \sum_{k=0}^n \binom{n}{k}^2 x^{k}$.
% \item[(2)] $\ds\CatD{n+1}{x} = \sum_{k=0}^{n} \binom{n}{k} \binom{n+1}{k} x^k$.
% \end{enumerate}
\end{proposition}

\def\Nar{\mathrm{Nar}}

\begin{remark}
This result shows that $\CatB{n}{x}$ is the Narayana polynomial of type $B_n$; i.e., the rank generating function of the lattice of noncrossing partitions of type $B_n$ introduced in \cite{Reiner_NC}.  By contrast, $\CatD{n}{x}$ is \emph{not} the Narayana polynomial of type $D_n$; however, the latter polynomial is equal to $\CatD{n}{x} + x\CatB{n-1}{x} - x\Cat{n-1}{x}$ for $n\geq 1$  (see \cite[Theorem 1.2]{typeD}). These observations imply that 
%\[
$\CatB{n}{1} = \binom{2n}{n}$
and %\qquad\text{and}\qquad 
$\CatD{n+1}{1} = \binom{2n+1}{n}$
%\] 
are central binomial coefficients.
\end{remark}

\section{Identities in types B and D}
\label{last}

%\begin{proof}[Proof of Proposition \ref{catb-prop}]
%Given $\Lambda \in \tNCB(n)$, define $\Lambda' \vdash [\pm n]$ by modifying $\Lambda$ in the following way:  first replace each pair of crossing arcs $(-i,j),(-j,i) $ in $ \Arc(\Lambda)$ with the pair of arcs $(-i,i),(-j,j)$, and then  
%remove zero from its block in the resulting set partition (removing this block if it is now empty). For example,
%\[ ???\]
%The correspondence $\Lambda \mapsto \Lambda'$ defines a bijection from $\tNCB(n)$ to the set of noncrossing set partitions $\Gamma\vdash[\pm n]$ with $-\Gamma = \Gamma$, and it is clear that if $\Lambda$ has $2k$ arcs then $\Lambda'$ has $2k$ or $2k-1$ arcs, and hence $2(n-k)$ or $2(n-k)+1$ blocks. 
%It follows from our discussion of the map $\nc$ that the the elements of $\tNCB(n)$ with $2k$ arcs are in bijection with the elements of $\NNB(n)$ with $2(n-k)$ or $2(n-k)+1$ blocks.  
%By Lemma \ref{obs1}, the number of elements of $\tNCB(n)$ with $2k$ arcs is therefore $\binom{n}{k}^2$, which suffices to prove part (1).
%%
%Part (2) following directly from (1) given  Theorem \ref{hanging-thm}.
%\end{proof}

To apply the methods of Section \ref{classical} 
to our new constructions, we  begin by defining the appropriate analogue of the map $\shift : \sP(n,\AA) \to \sP(n+1,\AA)$.

Shifting the matrix of a set partition in $\sPB(n,\AA)$ (respectively, $\sPD(n,\AA)$) one column to the right yields the matrix of a set partition in $\sPD(n+1,\AA)$ (respectively, $\sPB(n,\AA)$), and corresponds to two injective maps which, with slight abuse of notation, we again denote by $\shift$:
\[\shift: \sPD(n,\AA) \to \sPB(n,\AA)\qquad\text{and}\qquad \shift: \sPB(n,\AA) \to \sPD(n+1,\AA).\]
%both of which, with slight abuse of notation, we again denote by $\shift$.  
Explicitly:
\begin{enumerate}
\item[$\bullet$] If $\Lambda \in \sPD(n,\AA)$ then $\shift(\Lambda) \in \sPB(n,\AA)$ has arc set  $\{ (f(i), f(j)+1) : (i,j) \in \Arc(\Lambda)\}$ and labeling map 
$(f(i),f(j)+1) \mapsto \Lambda_{ij}$, where $f(x) \omdef = \begin{cases} x,&\text{if }x< 0, \\ x-1,&\text{if $x> 0$}.\end{cases}$

\item[$\bullet$]  If $\Lambda \in \sPB(n,\AA)$ then $\shift(\Lambda) \in \sPD(n+1,\AA)$ has arc set  
$\{ (g(i), g(j) + \delta_{j} + 1) : (i,j) \in \Arc(\Lambda)\}$ and labeling map 
$(g(i), g(j) + \delta_{j} + 1) \mapsto \Lambda_{ij}$, where $g(x) \omdef = \begin{cases} x-1,&\text{if $x\leq 0$}, \\ x,&\text{if $x>0$}.\end{cases}$  

\end{enumerate}
One checks without difficulty that these definitions do in fact give set partitions belonging to $\sPB(n,\AA)$ and $\sPD(n+1,\AA)$.  For example,
\[ 
\shift\( \xy<0.0cm,-0.0cm> \xymatrix@R=-0.0cm@C=.1cm{
*{\bullet} \ar @/^0.5pc/ @{-} [r]^a &
*{\bullet} \ar @/^.9pc/ @{-} [rr]^b &
*{\bullet} \ar @/^.9pc/ @{-} [rr]^{-b}  &
*{\bullet}   &
*{\bullet} \ar @/^0.5pc/ @{-} [r]^{-a}&
*{\bullet} 
\\
-3 &
-2   & 
-1 &
+ 1 &
+ 2 &
+ 3 
}\endxy\)
=
 \xy<0.0cm,-0.0cm> \xymatrix@R=-0.0cm@C=.1cm{
*{\bullet} \ar @/^0.9pc/ @{-} [rr]^a &
*{\bullet} \ar @/^1.2pc/ @{-} [rrr]^b &
*{\bullet} \ar @/^1.2pc/ @{-} [rrr]^{-b}  &
*{\bullet}   &
*{\bullet} \ar @/^0.9pc/ @{-} [rr]^{-a}&
*{\bullet} &
*{\bullet} 
\\
-3 &
-2   & 
-1 &
0 &
+ 1 &
+ 2 &
+ 3 
}\endxy
\] 
and
\[
\shift\( 
\xy<0.0cm,-0.0cm> \xymatrix@R=-0.0cm@C=.1cm{
*{\bullet} \ar @/^0.9pc/ @{-} [rr]^a &
*{\bullet} \ar @/^1.2pc/ @{-} [rrr]^b &
*{\bullet} \ar @/^1.2pc/ @{-} [rrr]^{-b}  &
*{\bullet}   &
*{\bullet} \ar @/^0.9pc/ @{-} [rr]^{-a} &
*{\bullet} &
*{\bullet} 
\\
-3 &
-2   & 
-1 &
0 &
+ 1 &
+ 2 &
+ 3 
}\endxy
\)
=
\xy<0.0cm,-0.0cm> \xymatrix@R=-0.0cm@C=.1cm{
*{\bullet} \ar @/^1.2pc/ @{-} [rrr]^a &
*{\bullet} \ar @/^1.5pc/ @{-} [rrrr]^b &
*{\bullet} \ar @/^1.5pc/ @{-} [rrrr]^{-b}  &
*{\bullet}   &
*{\bullet} \ar @/^1.2pc/ @{-} [rrr]^{-a} &
*{\bullet} &
*{\bullet} &
*{\bullet} 
\\
-4 &
-3   & 
-2 &
-1 &
+ 1 &
+ 2 &
+ 3 &
+4
}\endxy
\]
which becomes obvious after noting that the matrices of the three set partitions are
\[ \barr{|c|c|c|c|c|c|}
\hline
0 & a & & & & \\\hline
 & 0 & &  b & & \\\hline
  & & 0 &  & -b &\\\hline
  & & & 0&& \\\hline
  & & & & 0 & -a \\\hline
  & & & & & 0\\\hline
  \earr \quad
   \barr{|c|c|c|c|c|c|c|}
\hline
0 &  & a & & & & \\\hline
 & 0 &  & &  b & & \\\hline
 & & 0&  &  & -b &\\\hline
 & & & 0& && \\\hline
&  & & &0 &  & -a \\\hline
&  & & & &0 & \\\hline
&  & & & & & 0\\\hline
  \earr \quad
     \barr{|c|c|c|c|c|c|c|c|}
\hline
0& &  & a & & & & \\\hline
&0 &  &  & &  b & & \\\hline
 &&0 & &  &  & -b &\\\hline
 & & &0& & && \\\hline
&  & & &0& &  & -a \\\hline
&  & & & &0& & \\\hline
&  & & & & &0& \\\hline
&  & & & & && 0\\\hline
  \earr
  \]

 It is straightforward to see that the map $\shift$ defines bijections
\be\label{b-shift} \ba &\sPD(n,\AA) \to \{ \text{2-regular elements of $\sPB(n,\AA)$}\},
\\
&
\sPB(n,\AA) \to \{ \text{2-regular elements of $\sPD(n+1,\AA)$}\},
\ea
\ee where, as previously, a set partition is 2-regular if  none of its blocks contain two consecutive integers $i$ and $i+1$.  
Consequently we may view $\shift$ as a map $\sPD(n,\AA) \to\sPB(n,\AA,\BB)$ and $\sPB(n,\AA) \to \sPD(n+1,\AA,\BB)$.

We define the feasible and poor elements of 
$\sPB(n)$ or $\sPD(n)$  exactly as for $\sP(n)$.  In addition, we say that a set partition is \emph{B-feasible} if none of its blocks contain exactly one nonzero element and \emph{B-poor} if none of its blocks contain more than two nonzero elements.
Observe that these notions are distinct from ``feasible'' and ``poor'' only for elements of $\sPB(n)$.
The following lemma, in analogy with Lemma \ref{shift-bij}, describes the action of $\shift$ on these sets of interest.

\begin{lemma}\label{b-shift-bij} The following restrictions of $\shift$ are bijections:
\begin{enumerate}
\item[(1)] $ \left\{ \text{Feasible elements of $\sPD(n,\AA)$} \right\} \to \left\{ \ba &\text{2-regular partitions $\Lambda \in \sPB(n,\AA)$ such that} \\&\text{$1+\max B \neq \min B'$ for all blocks $B,B' \in \Lambda$}\ea \right\}$.

\item[(2)] $\left  \{ \text{B-feasible elements of $\sPB(n,\AA)$} \right\} \to \left\{ \ba &\text{\small 2-regular  partitions $\Lambda\in \sPD(n+1,\AA)$ with} \\&\text{\small $1+\max B \neq \min B'$ for all blocks $B,B' \in \Lambda$}\ea \right\}$.

\item[(3)]
$\left\{ \text{Poor elements of $\tNCD(n,\AA)$} \right\}  \to \left\{ \text{2-regular elements of $\tNCB(n,\AA)$} \right\}$.

\item[(4)] 
$\left\{ \text{B-poor elements of $\tNCB(n,\AA)$} \right\}  \to \left\{ \text{2-regular elements of $\tNCD(n+1,\AA)$}\right\}$.

\end{enumerate}
\end{lemma}

\begin{proof} 
Parts (1) and (2) follow from the intuitive definition of $\shift$ on matrices after noting that
\begin{enumerate}
\item[$\bullet$] A partition $\Lambda$ in $\sPD(n,\AA)$ (respectively, $\sPB(n,\AA)$) is feasible (respectively, B-feasible) if and only if its matrix has a nonzero entry in either the $i$th row or $i$th column for each $i \in [2n]$ (respectively, for each $i \in [2n+1]\setminus\{n+1\}$).

\item[$\bullet$]
A partition $\Lambda $ in  $\sPB(n,\AA)$ (respectively, $\sPD(n+1,\AA)$) is 2-regular and has the property that $1+\max B \neq \min B'$ for all blocks $B,B' \in \Lambda$ if and only if its matrix has no nonzero entries on the superdiagonal but has at least one nonzero entry in the superdiagonal hook $\{ (i,j+1) : i < j\} \cup \{ (j,k) : j<k \}$ for each $i \in [2n]$ (respectively, for each $i \in [2n+2] \setminus \{ n+1\}$).  

\end{enumerate}
Parts (3) and (4) follow from similar considerations.
%Similarly, parts (3) and (4) follow by noting that 
%\begin{enumerate}
%\item[$\bullet$] A partition $\Lambda$ in $\tNCD(n,\AA)$ (respectively, $\tNCB(n,\AA)$) is poor (respectively, B-poor) if and only if its matrix has at most one nonzero entry in  the diagonal hook $\{ (i,j) : i < j\} \cup \{ (j,k) : j<k \}$ for each $i \in [2n]$ (respectively, for each $i \in [2n+1] \setminus \{ n+1\}$). 
%
%\item[$\bullet$]
%A partition $\Lambda $ in  $\tNCB(n,\AA)$ (respectively, $\tNCD(n+1,\AA)$) is 2-regular  if and only if its matrix has no nonzero entries on the superdiagonal and has at most one nonzero entry in the superdiagonal hook $\{ (i,j+1) : i < j\} \cup \{ (j,k) : j<k \}$ for each $i \in [2n]$ (respectively, for each $i \in [2n+2] \setminus \{ n+1\}$).  
%\end{enumerate}
\end{proof}

For $\X \in \{\mathrm{B}, \mathrm{D}\}$, we let $\FeX{n}{x}$, $\wFeB{n}{x}$, $\MoX{n}{x}$, and $\wMoB{n}{x}$ denote the polynomials
\[
\begin{aligned}
 \ds\FeX{n}{x} &= \sum_{\text{Feasible }\Lambda \in \sPX(n)} x^{\frac{|\Arc(\Lambda)|}{2}}
 \\[-10pt] \\
\ds \wFeB{n}{x} &= \sum_{\text{B-feasible }\Lambda \in \tNCB(n)} x^{\frac{|\Arc(\Lambda)|}{2}}
\end{aligned}
\qquad\text{and}\qquad
\begin{aligned}
\ds \MoX{n}{x} &=\sum_{\text{Poor }\Lambda \in \tNCX(n)} x^{\frac{|\Arc(\Lambda)|}{2}}
 \\[-10pt] \\
\ds  \wMoB{n}{x} &= \sum_{\text{B-poor }\Lambda \in \tNCB(n)} x^{\frac{|\Arc(\Lambda)|}{2}}.
  \end{aligned}
 \]
%We also define
%\[ \qquad\text{and}\qquad
% \wMoB{n}{x} = \sum_{\text{B-poor }\Lambda \in \tNCB(n)} x^{\frac{|\Arc(\Lambda)|}{2}}.\]
 The next few results provide more explicit formulas for these functions.

To begin, observe that in analogy with (\ref{belld-eq}), we have
\be \FeD{n}{x} = \Fe{n}{2x},\qquad\text{for all $n\geq 0$}.\ee
In a different direction, note that a poor set partition in $\sPB(n)$ must contain the singleton $\{0\}$ as a block, and removing this block defines a bijection from the set of poor elements of $\sPB(n)$ to the set of poor elements of  $\sPD(n)$.
The first equality in the next proposition derives from this fact.

\begin{proposition}\label{motzkinb-prop} $\ds\MoB{n}{x} =\MoD{n}{x}= \sum_{k=0}^{\lfloor n/2 \rfloor} \binom{2k}{k} \binom{n}{2k} x^k$.
\end{proposition}

\begin{remark} The numbers $\{\MoB{n}{1}\}_{n=0}^\infty = (1, 1, 3, 7, 19, 51, 141,\dots)$ are the central trinomial coefficients \cite[A002426]{OEIS}, defined as the  coefficient of $x^n$ in  $(1+x+x^2)^n$.
\end{remark}

\begin{proof}
Since a poor element of $\sPD(n)$ with $2k$ arcs has $2(n-k)$ singleton blocks which come in pairs $\{ i\}, \{-i\}$, we have $\MoD{n}{x} = \sum_k \binom{n}{k} |\cX_k| x^k$ where $\cX_k$ is the set of poor partitions in $\tNCD(k)$ whose blocks all have size two.
By Lemma \ref{uncross-obs}, the map $\nc$ defines a bijection 
$\cX_k \to \cY_k$, where $\cY_k$ is the set of 
noncrossing set partitions $\Lambda\vdash [\pm k]$ with $\Lambda =-\Lambda$ which have $k$ blocks of size two and an even number of blocks of the form $\{-i,i\}$.  To prove the proposition it  suffices to show that $\cY_k$ is empty if $k$ is odd and that $|\cY_{2k}| = \binom{2k}{k}$. 

 Since the elements of $\cY_k$ are noncrossing and invariant under negation, the blocks of  a partition $\Lambda \in \cY_k$ which are not of the form $\{-i,i\}$ are of the form $\{i,j\}$ with $i,j > 0$ or $i,j < 0$.  Hence, removing all blocks of $\Lambda \in \cY_k$ which contain negative integers (and then shifting indices) produces a noncrossing partition of $[k-\ell]$ whose blocks all have size two, where $\ell$ is the number of blocks of $\Lambda$ of the form $\{ -i,i\}$.  Since $\ell$ is always even, if $k$ is odd then no such partitions exist and $|\cY_k| = 0$.  
 
Let $\cZ_k$ be the set of all noncrossing set partitions of $[2k]$ whose blocks all have size two.  Given $\Lambda \in \cZ_k$, let $\varphi(\Lambda) \in \cD_k$ be the Dyck path whose $i$th step is $(1,1)$ if $i$ is the smaller of the two elements in its part of $\Lambda$ and $(1,-1)$ otherwise.  One checks that $\varphi : \cZ_k \to \cD_k$ is a well-defined bijection (one recovers $\Lambda$ by numbering the 
 steps of $\varphi(\Lambda)$ from 1 to $2k$ and placing the pairs of numbers indexing each upstep $(1,1)$ and the following downstep $(1,-1)$ at the same height in blocks), and it is clear that if we view $\cY_{2k}$ as a subset of $\cZ_{2k}$, then $\varphi$ restricts to a bijection from $\cY_{2k}$ to the set of symmetric elements of $\cD_{2k}$.  Hence $|\cY_{2k}| = \binom{2k}{k}$ by Lemma \ref{obs2}, as required.
\end{proof}

Inspecting the formulas for $\Mo{n}{x}$ and $\MoB{n}{x}$ in Propositions \ref{motzkin-obs} and  \ref{motzkinb-prop} yields the following trivial but useful corollary.

\begin{corollary}\label{mob-cor} $\ds\MoB{n+2}{x} =\MoB{n+1}{x} + 2(n+1)x \Mo{n}{x}$ for $n \geq 0$.
%\begin{enumerate}
%\item[(1)] $\ds\MoB{n+2}{x} =\MoB{n+1}{x} + 2(n+1)x \Mo{n}{x}$.
%\item[(2)] $\ds \sum_{k=0}^n \Mo{k}{x}  \MoB{n-k}{x} = (n+1) \Mo{n}{x}$.
%\end{enumerate}
\end{corollary}

%\begin{proof}
%Part (1) follows by inspecting the formulas for $\Mo{n}{x}$ and $\MoB{n}{x}$ in Observation \ref{motzkin-obs} and Proposition \ref{motzkinb-prop}.  \end{proof}

We now turn our attention to the polynomials $\wFeB{n}{x}$ and $\wMoB{n}{x}$ counting B-feasible and B-poor  noncrossing partitions in $\sPB(n,\AA)$. 

\begin{proposition}\label{tilde-prop} The following identities hold for $n\geq 0$:
\begin{enumerate}
\item[(1)] $\ds\wFeB{n}{x} =\FeB{n}{x} + \Fe{n}{2x} = 
\sum_{k=0}^n \binom{n}{k} \Fe{k}{2x} x^{n-k}$.
\item[(2)] $\ds\wMoB{n}{x}  = \MoB{n}{x} + nx \Mo{n-1}{x}=\sum_{k=0}^{\lceil n/2\rceil}  \binom{n}{k} \binom{n+1-k}{k}
 x^k$.
 \end{enumerate}
\end{proposition}

\begin{remark} We have $\wFeB{n}{1} = \Bell{n}{2}$  since splitting the block containing $0$ into singletons and then removing $\{0\}$ defines a bijection from the set of B-feasible elements of $\sPB(n)$ to $\sPD(n)$; this also follows from the first remark to Theorem \ref{identities}.
The numbers $\{\wMoB{n}{1}\}_{n=0}^\infty = (	1, 1, 2, 5, 13, 35, 96,\dots)$ count the directed animals of size $n+1$ \cite[A005773]{OEIS}.  \end{remark}

\begin{proof}
A B-feasible element of $\sPB(n)$ is either feasible or contains $\{0\}$ as a block, and removing the block $\{0\}$ gives a bijection from the partitions of the latter kind to the feasible elements of $\sPD(n)$.  We obtain the second equality in (1) by noting that $\binom{n}{k} \Fe{k}{2x}x^{n-k}$ is the number of B-feasible elements of $\sPB(n)$ in which 0 belongs to a block with exactly $2k+1$ elements. %noting that $\FeB{n}{x} + \Fe{n}{2x} = \sum_{\ell=0}^n \sum_{k=0}^\ell \binom{n}{\ell} \stst{\ell}{k} 2^{\ell-k} x^{n-k}$ by  (\ref{stirlstirlb-prop}).

For part (2), we observe that $\wMoB{n}{x} = \MoB{n}{x} + \sum_{k=1}^n \wMoB{n,k}{x}$ where $ \wMoB{n,k}{x}$ is the sum of $x^{\frac{|\Arc(\Lambda)|}{2}}$ over all B-poor set partitions $\Lambda \in \tNCB(n)$ which possess $\{-k,0,k\}$ as a block.  In such partitions, all remaining blocks are subsets of either $\{1,2,\dots,k-1\}$, $\{-1,-2,\dots,-k+1\}$, or $[\pm n] \setminus [\pm k]$.  The blocks contained in the first two of these sets are reflections of each other and determine a poor element of $\NC(k-1)$, while the blocks contained in $[\pm n] \setminus [\pm k]$ determine a poor element of $\tNCD(n-k)$.  It follows that $\wMoB{n,k}{x} =x \Mo{k-1}{x}\MoB{n-k}{x}$.

By considering whether the blocks of $\pm (n+2)$ in a poor element of $\tNCD(n+2,\AA)$ are singletons or contain a second element in $[\pm(n+1)]$, one obtains the recurrence
$\MoB{n+2}{x} = \MoB{n+1}{x} + 2x\sum_{k=0}^{n} \Mo{k}{x}\MoB{n-k}{x}$ for $n\geq 0$. Subtracting  the expression in Corollary \ref{mob-cor} from this equation shows that 
$\sum_{k=1}^n \wMoB{n,k}{x} = n x \Mo{n-1}{x}$, which 
then  gives the formula in (2).
% we conclude that
%\[\wMoB{n+1}{x} = \MoB{n+1}{x} + (n+1)x \Mo{n}{x},\qquad\text{for }n\geq 0.\]
%The given formula follows directly.
\end{proof}

%Exactly as for $\BellB{n+1}{x}$, by counting the feasible elements of $\sPB(n+1,\AA)$ according to whether $n+1$ belongs to the same block as zero or to a block contained in $[\pm (n+1)]$, one obtains the recurrence
%\be\label{tilde-fe-rec}\wFeB{n+1}{x} = x\wFeB{n}{x}  + \sum_{k=0}^{n-1} \binom{n}{k} \wFeB{k}{x} (2x)^{n-k},\qquad\text{for $n\geq 0$, with $\wFeB{0}{x}= 1$.}\ee
%Subtracting (\ref{fe-rec}) from this gives a similar formula for $\FeB{n+1}{x}$.

The following corollary %to Propositions \ref{motzkinb-prop} and \ref{tilde-prop}
 will be of use in the proof of Theorem \ref{b-identities}.  

\begin{corollary}\label{2blocks} The following identities hold for $n\geq 0$:
\begin{enumerate}
\item[(1)]
There are $\binom{2n}{n}$  elements of $\tNCD(2n)$ whose blocks all have size two.

\item[(2)] There are no elements of $\tNCD(2n+1)$ whose blocks all have size two. 

\item[(3)]
There are $\binom{n}{\lfloor n/2\rfloor}$  B-poor elements of $\tNCB(n)$ with no nonzero singleton blocks.
\end{enumerate}
\end{corollary}

%\begin{remark}
%Multiplying part (1) by  $\binom{m+2n}{m}$ gives  the number of poor elements of $\tNCD(m+n)$ and B-poor elements of $\tNCB(m+n)$ with exactly $2m$ nonzero singleton blocks.
%\end{remark}

\begin{proof}
All blocks of an element of $\tNCD(n)$ have size two if and only if the partition has exactly $n$ arcs; hence, 
the number counted in part (1) is the coefficient of $x^{n/2}$ in $\MoD{n}{x}$.
An element of $\tNCB(n)$ is B-poor and has no nonzero singleton blocks 
if and only if 
either 
the partition contains $\{ 0\}$ as a block such that removing this block produces an element of $\tNCD(n)$ whose blocks all have size two, or if the partition has $n+1$ arcs.
Hence the 
 number counted in part (2) is the sum of the number in (1) and the coefficient of $x^{(n+1)/2}$ in $\wMoB{n}{x}$.
\end{proof}

Serving as an analogue for Theorem \ref{main-thm},  the following result 
 shows that 
$\Bell{n}{2x}$  and $\MoB{n}{x}$ count the  $\LinB(n,\BB)$-orbits in $\sPB(n,\AA,\BB)$ and $\tNCB(n,\AA,\BB)$, while 
$\BellB{n}{x}$  and $\wMoB{n}{x}$ count the  $\LinD(n+1,\BB)$-orbits in $\sPD(n+1,\AA,\BB)$ and $\tNCD(n+1,\AA,\BB)$.

\begin{theorem}\label{b-thm} Let $n$ be a positive integer.
\begin{enumerate}
\item[(1)]
 The correspondence $\Lambda \mapsto \{ \alpha + \shift(\Lambda) : \alpha \in \LinB(n,\BB)\}$
  is a bijection 
\[ \ba \sPD(n,\AA)&\to
\left\{  \LinB(n,\BB)
\text{-orbits in }\sPB(n,\AA,\BB)\right\},
  \\
 \left \{ \text{Poor elements of $\tNCD(n,\AA)$}\right\} &\to  
\left\{ \LinB(n,\BB)\text{-orbits in } \tNCB(n,\AA,\BB)\right\}.
   \ea
  \] 
  Furthermore,  the cardinality of the $\LinB(n,\BB)$-orbit of $\shift(\Lambda)$ is $|\BB|^{s/2}$, where $s$ is the number of singleton blocks of $\Lambda \in \sPD(n,\AA)$.
  
  \item[(2)] 
   The correspondence $\Lambda \mapsto \{ \alpha + \shift(\Lambda) : \alpha \in \LinD(n,\BB)\}$
  is a bijection 
\[ \ba \sPB(n-1,\AA)&\to
\left\{  \LinD(n,\BB)
\text{-orbits in }\sPD(n,\AA,\BB)\right\},
  \\
 \left \{ \text{B-poor elements of $\tNCB(n-1,\AA)$}\right\} &\to  
\left\{ \LinD(n,\BB)\text{-orbits in } \tNCD(n,\AA,\BB)\right\}.
   \ea
  \] 
The cardinality of the $\LinD(n,\BB)$-orbit of $\shift(\Lambda)$ is $|\BB|^{\lfloor s/2\rfloor }$, where $s$ is the number of singleton blocks of $\Lambda \in \sPB(n-1,\AA)$.

% other than $\{0\}$ of $\Lambda \in \sPB(n-1,\AA)$.  
  \end{enumerate}
%  Consequently, if $|\AA| = x+1$ and $|\BB| = y+1$ and $\Lambda \in \sPD(n,\AA)$ has $s$ singleton blocks, then
%  \begin{enumerate}
%  \item[(1)] There are $\Bell{n}{2x}$   distinct $\LinB(n,\BB)$-orbits in $\sPB(n,\AA,\BB)$;
%  \item[(2)] There are $\MoD{n}{x}$ distinct $\LinB(n,\BB)$-orbits in $\tNCB(n,\AA,\BB)$;
%  \item[(3)] The $\LinB(n,\BB)$-orbit of $\shift(\Lambda)$ has cardinality $(y+1)^{s/2}$.  \end{enumerate}
\end{theorem}

In part (2), $\lfloor s/2\rfloor$  is half the number of \emph{nonzero} singleton blocks of $\Lambda \in \sPB(n-1,\AA)$.

\begin{proof}
The proof is quite similar to that of Theorem \ref{main-thm}.  Let $\X \in \{ \mathrm{B}, \mathrm{D}\}$. As in the earlier proof, the definition of $+$ implies that each $\LinX(n,\BB)$-orbit in $\sPX(n,\AA,\BB)$ contains a unique 2-regular element, and this suffices by (\ref{b-shift}) and Lemma \ref{b-shift-bij} to show that the maps in (1) and (2) are bijections.  

If $\alpha \in \LinX(n,\BB)$ then $(i,i+1) \in \Arc(\alpha)$ if and only if $(-i-1,-i) \in \Arc(\alpha)$ and the pair of arcs $(i,i+1)$, $(-i-1,-i)$ have opposite labels.  Also, if $\alpha \in \LinD(n,\BB)$ then $(-1,1)$ never belongs to $\Arc(\alpha)$.
Consequently, from our definition of $+$ via Observation \ref{+def2}, it follows that if $\Lambda$ belongs to  $\sPD(n,\AA)$ (respectively, $\sPB(n-1,\AA)$) then 
the orbit of $\shift(\Lambda)$ has size $|\BB|^{s/2}$, where $s$ is the  number the superdiagonal hooks $\{ (i,j+1) : i<j\} \cup \{(j,k) : j<k\}$ for $j \in [2n]$ (respectively, for $j \in [2n-1]\setminus\{n\}$) containing no nonzero entries in the matrix of $\shift(\Lambda)$. In both cases, the intuitive definition of $\shift$ implies that $s$ is the (even) number of nonzero singleton blocks in $\Lambda$.
\end{proof}

We may now prove the type B and D analogue of Theorem \ref{identities}.
%Now, matching Theorem \ref{identities}, we derive two formulas each for the polynomials $\BellB{n}{x,y}$, $\BellD{n}{x,y}$, $\CatB{n}{x,y}$, and $\CatD{n}{x,y}$.

\begin{theorem}\label{b-identities}
If $n$ is a nonnegative integer then the following identities hold:
\begin{itemize}
\item[(1)]
$\BellB{n}{x,y} =\ds \sum_{k=0}^n \binom{n}{k} \Bell{k}{2x} y^{n-k} =\sum_{k=0}^n \binom{n}{k} \Fe{k}{2x} (y+1)^{n-k}$.

\item[(2)]
$  \BellD{n+1}{x,y} = \ds\sum_{k=0}^n \binom{n}{k} \BellB{k}{x} y^{n-k} = \sum_{k=0}^n \binom{n}{k} \wFeB{k}{x} (y+1)^{n-k}$.

\item[(3)]
$\ds \CatB{n}{x,y} = \sum_{k=0}^n \binom{n}{k} \MoB{k}{x} y^{n-k} = \sum_{k=0}^n \binom{2k}{k}  \binom{n}{2k}x^k (y+1)^{n-2k} $.

\item[(4)]
$\ds \CatD{n+1}{x,y} = \sum_{k=0}^n \binom{n}{k} \wMoB{k}{x} y^{n-k} = \sum_{k=0}^n \binom{n}{k} \binom{k}{\lfloor k/2\rfloor} x^{\lceil k/2\rceil} (y+1)^{n-k} $.

\end{itemize}
\end{theorem}

\begin{proof}
As in the proof of Theorem \ref{identities}, if $|\AA| = x+1 $ and $|\BB| = y+1$, then in each part, the terms in first sum counts partitions (in $\sPB(n,\AA,\BB)$, $\sPD(n+1,\AA,\BB)$, $\tNCB(n,\AA,\BB)$, or $\tNCD(n+1,\AA,\BB)$, respectively) whose  matrices have  the same number of %$2(n-k)$
 nonzero entries on the superdiagonal, while the terms in the second sum count partitions whose $\LinB(n,\BB)$- or $\LinD(n+1,\BB)$-orbits have the same fixed size. 
 %, given by 
%\[\begin{cases} (y+1)^{n-k}&\text{in parts (1), (2), and (4)}, \\ (y+1)^{n-2k}&\text{in part (3).}\end{cases}\]  
Checking the details of this assertion$-$using Corollary \ref{2blocks}, Equation (\ref{b-shift}), Lemma \ref{b-shift-bij}, and Theorem \ref{b-thm}$-$is entirely analogous to the proof of Theorem \ref{identities}, and we leave this exercise to the reader.
The sums in each part are thus necessarily equal to each other and to $\BellB{n}{x,y}$ in (1), $\BellD{n+1}{x,y}$ in (2), $\CatB{n}{x,y}$ in (3), and $\CatD{n+1}{x,y}$ in (4).
\end{proof}

\begin{remark} 
%A few comments are forthcoming.
%\begin{enumerate}
%\item[(i)] Setting $y=2x$ and $y=0$ in part (1) and using Theorem \ref{identities} yields the identity \[  \BellB{n}{x,2x} = \BellB{n+1}{x,0} = \Bell{n+1}{2x}
% = \sum_{k=0}^n \binom{n}{k}\BellB{k}{x} x^{n-k} 
%\\&
%= \sum_{k=0}^n \binom{n}{k} \wFeB{k}{x}(x+1)^{n-k}
%  \\&
% = \sum_{k=0}^n \binom{n}{k}\Fe{k}{2x} (2x+1)^{n-k}
% \\&
%  =  \sum_{k=0}^n\binom{n}{k} \Bell{k}{2x} (2x)^{n-k} 
%  .\]
%Setting $y=0$ in part (2) shows that $\BellB{n}{x}= \sum_k \binom{n}{k} \wFeB{k}{x}$. Combined with (\ref{tilde-fe-rec}), this leads to the equation $\BellB{n}{\frac{1}{2}} =\frac{1}{2} \wFeB{n}{\frac{1}{2}} + \wFeB{n+1}{\frac{1}{2}}$, which one might view as a type B analogue of the identity $\Bell{n}{1} = \Fe{n}{1} + \Fe{n+1}{1}$ noted in the remarks to Theorem \ref{identities}.  
%Since $\BellB{n}{x} = x^n D_2(n,1/x)$ in the notation of \cite{Dowling2}, it follows from \cite[Proposition 1]{Dowling2} that $\wFeB{n}{x}$ has the generating function
%\[
%%\sum_{n\geq 0} \BellB{n}{x} \frac{z^n}{n!} = \exp\(xz + \tfrac{\exp(2xz)-1}{2x}\)
%%\quad\text{and}\quad
%\sum_{n\geq 0} \wFeB{n}{x} \frac{z^n}{n!} = \exp\((x-1)z + \tfrac{\exp(2xz)-1}{2x}\).\] 
%Taking logarithms and differentiating leads to the recurrence
%\[\wFeB{n+1}{x} = x\wFeB{n}{x}  + \sum_{k=0}^{n-1} \binom{n}{k} \wFeB{k}{x} (2x)^{n-k},\qquad\text{for $n\geq 0$, with $\wFeB{0}{x}= 1$.}\]
%which implies by ???? that 
%\[\FeB{n+1}{x} = (x-1)\FeB{n}{x} +  x \Fe{n}{2x} + \sum_{k=0}^{n} \binom{n}{k} \FeB{k}{x} (2x)^{n-k},\qquad\text{for $n\geq 0$.}\]

%\item[(ii)]
 We may rewrite   part  (3) as the following identity, first proved in a different way by Chen, Wang, and Zhao  \cite[Theorem 2.5]{CWZ}:
\[ \sum_{i=0}^{n} \sum_{j=0}^{\lfloor \frac{n-i}{2} \rfloor} \binom{n}{i} \binom{n-i}{j} \binom{n-i-j}{j} x^{j} y^i
=
\sum_{k=0}^{\lfloor n/2\rfloor}  \binom{2k}{k}  \binom{n}{2k} x^k (y+1)^{n-2k}.\]
 Setting $x=y$ here recovers the identity (\ref{intro2}) mentioned in the introduction, and taking $x=y=1$  gives $\binom{2n}{n} = \sum_k \binom{n}{2k} \binom{2k}{k} 2^{n-2k}$, a type B analogue   for Touchard's formula for the Catalan numbers noted by Simion \cite{typeB}.

%\item[(3)] Substituting  the expressions for $\Mo{k}{x}$ and $\wMoB{k}{x}$ given in Corollary \ref{mob-cor} and Proposition \ref{tilde-prop} into parts (3) and (4) while noting the second part of Theorem \ref{identities} shows  that 
%\be\label{hanging}\ba
%\CatB{n+1}{x,y} &= (y+1)\CatB{n}{x,y} + 2nx \Cat{n}{x,y},\\
%%\item[(2)] $ \ds\sum_{k=0}^n \Cat{n-k}{x,y}  \CatB{k}{x,y} = n \Cat{n}{x,y} + \CatB{n}{x,y}$
%\CatD{n+1}{x,y} &= \CatB{n}{x,y} + nx\Cat{n}{x,y}.
%\ea
%%\qquad\text{for }n\geq 0,
%\ee
%The second equation shows that many of our type D identities are in fact redundant.
% For example, the equality of the two expressions in part (4) follows by taking an appropriate linear combination of part (2) of Theorem \ref{identities} and part (3) of this theorem.  We may rewrite the first equation in (\ref{hanging}) as
%\[\CatB{n+1}{x,y} =   (y+1)^{n+1} + 2x \sum_{k=0}^n k(y+1)^{n-k} \Cat{k}{x,y} .\]  Setting $x=y=1$ gives the identity $\ds\binom{2n}{n} = 2^n + \sum_{k=1}^{n-1} 2^{n-k} k \cC_k$ relating the central binomial coefficients and the Catalan numbers.
%\end{enumerate}

\end{remark}

Substituting  the expressions for $\Mo{k}{x}$ and $\wMoB{k}{x}$ given in Corollary \ref{mob-cor} and Proposition \ref{tilde-prop} into parts (3) and (4) of the preceding result, while noting the second part of Theorem \ref{identities}, leads to the next corollary. 
    \begin{corollary}\label{hanging} The following identities hold for $n\geq 0$:
\begin{enumerate}
\item[(1)] $\CatB{n+1}{x,y} = (y+1)\CatB{n}{x,y} + 2nx \Cat{n}{x,y}$.
%\item[(2)] $ \ds\sum_{k=0}^n \Cat{n-k}{x,y}  \CatB{k}{x,y} = n \Cat{n}{x,y} + \CatB{n}{x,y}$
\item[(2)] $\CatD{n+1}{x,y} = \CatB{n}{x,y} + nx\Cat{n}{x,y}$.
\end{enumerate}
  \end{corollary}
This  brings to light some redundancy in our identities.
The first part provides a way to derive (\ref{intro1}) from (\ref{intro2}) in the introduction, and the second part shows that the  equality of the two expressions in part (4) of Theorem \ref{b-identities} follows by taking an appropriate linear combination of part (2) of Theorem \ref{identities} and part (3) of  Theorem \ref{b-identities}.

As a second corollary, we have this  analogue of  Corollary \ref{main-cor}. %, the following result employs the inclusion-exclusion principle to derive alternate formulas for these expressions.

\begin{corollary} \label{b-cor}
If  $n$ is a nonnegative integer then the following identities hold:
\begin{enumerate}
\item[(1)] 
$\ds\sum_{k=0}^n (-1)^{n-k} \binom{n}{k}  (y+1)^{n-k} \BellB{k}{x,y} =\Fe{n}{2x}$.

\item[(2)] $\ds\sum_{k=0}^n (-1)^{n-k} \binom{n}{k}  (y+1)^{n-k} \BellD{k+1}{x,y} =\wFeB{n}{x}$.

\item[(3)] 
$\ds\sum_{k=0}^n (-1)^{n-k} \binom{n}{k} (y+1)^{n-k}  \CatB{k}{x,y}= \begin{cases} \binom{n}{n/2} x^{n/2}, &\text{if $n$ is even,} \\
0,&\text{if $n$ is odd.}\end{cases}$

\item[(4)] $\ds\sum_{k=0}^n (-1)^{n-k} \binom{n}{k} (y+1)^{n-k}  \CatD{k+1}{x,y}= \binom{n}{\lfloor n/2\rfloor}x^{\lceil n/2\rceil} $.

\end{enumerate}

  \end{corollary}

\begin{remark}
Setting $x=y$ in  part (3) produces the result given by Chen, Wang, and Zhao as \cite[Theorem 3.1]{CWZ}. Setting $x=y=1$ in this equation gives an identity attributed to Dawson \cite[Page 71]{riordan}, for which Andrews \cite[Theorem 5.4]{andrews} gives a proof using basic hypergeometric functions. 
%Noting Corollary \ref{hanging}, one sees that part (4) is a linear combination of part (3) and the second part of Corollary \ref{main-cor}.
 \end{remark}

\begin{proof}
The proof is almost the same as that of Corollary \ref{main-cor}.  
In short, letting given a subset $S \subset [n]$, let
 $\cX_S$ and $\cY_S$ denote the unions of the $\LinB(n,\BB)$-orbits
 of $\shift(\Lambda)$ for all partitions $\Lambda$ in $\sPD(n,\AA)$ and $\tNCD(n,\AA)$, respectively,
which contain the singletons $\{i\}$ and $\{-i\}$ as blocks for each $i \in S$.
By Theorem \ref{b-thm} the sets of $\LinB(n,\BB)$-invariants in $\sPB(n,\AA,\BB)$ and $\tNCB(n,\AA,\BB)$ are the respective complements of $\bigcup_{i \in [n]}\cX_{\{ i\}}$ and 
$\bigcup_{i \in [n]}\cY_{\{ i\}}$. Using Theorems \ref{b-thm} and \ref{b-identities} and the inclusion-exclusion principle to count the elements in these unions affords parts (1) and (3) upon setting $x = |\AA|-1$ and $y = |\BB|-1$.
%%
%% It is straightforward to compute from Theorems \ref{b-thm} and \ref{b-identities} that
%%\[
%%|\cX_S| = (y+1)^{|S|} \BellB{n-|S|}{x,y}\qquad\text{and}\qquad
%%|\cY_S| = (y+1)^{|S|} \CatB{n-|S|}{x,y}.\]
%%The inclusion-exclusion principle then affords parts (1) and (3) since
%%by Theorem \ref{b-thm} the set of $\LinB(n,\BB)$-invariants in $\sPB(n,\AA,\BB)$ and $\tNCB(n,\AA,\BB)$ are the respective complements of $\bigcup_{i \in [n]}\cX_{\{ i\}}$ and 
%%$\bigcup_{i \in [n]}\cY_{\{ i\}}$.

One proves parts (2) and (4) in the same way,
by considering
the analogous sets $\cX'_S$ and $\cY'_S$ given by the unions of the $\LinD(n+1,\BB)$-orbits
 of $\shift(\Lambda)$ for all partitions $\Lambda$ in $\sPB(n,\AA)$ and $\tNCB(n,\AA)$, respectively,
which contain the singletons $\{i\}$ and $\{-i\}$ as blocks for each $i \in S$.
\end{proof}

We conclude this section by noting the following  recurrences for $\Cat{n}{x,y}$ and $\CatB{n}{x,y}$, which are easy consequences of the preceding results but which are difficult to deduce otherwise.

\begin{corollary} For $n\geq 2$, the following recurrences hold:
\begin{enumerate}
\item[(1)] $(n+1)\Cat{n}{x,y} =(y+1) (2n-1)  \Cat{n-1}{x,y} + \(4x-(y+1)^2\)(n-2)  \Cat{n-2}{x,y}.$

\item[(2)] $n \CatB{n}{x,y} =(y+1)(2n-1) \CatB{n-1}{x,y} + \(4x-(y+1)^2\)(n-1)\CatB{n-2}{x,y}.$

%\item[(3)] For $n \geq 2$,
%  \[ \ba (n+1) \CatD{n+1}{x,y} &= \Bigl( (3y+1)n+2\Bigr) \CatD{n}{x,y}  
% \\&\quad +  \Bigl( (4x-3y^2-2y+1)n+ 3(y^2-1)\Bigr) \CatD{n-1}{x,y} 
% \\&\quad+ \Bigl((y-1)\((y+1)^2 - 4x\)(n-2)\Bigr) \CatD{n-2}{x,y}.
% \ea
% \] 

\end{enumerate}
\end{corollary}

\begin{remark}
Sulanke provides bijective proofs of part (1) in the increasingly general  cases $x=y \in \{1,2\}$ \cite{SulankeSchroeder}; $x=1$ \cite[Proposition 1.1]{SulankeMotzkin}; and $x=y$ \cite{SulankeNarayana}.
\end{remark}

\begin{proof}
After substituting the right-most expressions in Theorem \ref{identities} and Theorem \ref{b-identities} for $\Cat{n}{x,y}$ and $\CatB{n}{x,y}$ into these equations, 
the theorem is equivalent to easily checked identities involving binomial coefficients and Catalan numbers.
%
%part (1) is equivalent to the statement that for $n\geq 2$ and $k\geq 0$,
%\[ (n+1) \binom{n-1}{2k}\cC_k = (2n-1) \binom{n-2}{2k}\cC_k  + 4 (n-2) \binom{n-3}{2k-2}\cC_{k-1}  - (n-2) \binom{n-3}{2k}\cC_k, \] where we define $\cC_{-1} =0$ and $\binom{n}{-2} = \binom{-1}{k} = 0$ for all $n,k$.  Checking this is routine.  Similarly, after substituting the right-most expression in Theorem  \ref{b-identities} for $\CatB{n}{x,y}$, part (2) is equivalent to the statement that for $n\geq 2$ and $k\geq 0$,
%\[ n\binom{n}{2k} \binom{2k}{k}=(2n-1) \binom{n-1}{2k}\binom{2k}{k} + 4(n-1)\binom{n-2}{2k-2} \binom{2k-2}{k-1} - (n-1) \binom{n-2}{2k} \binom{2k}{k},\]
%where by definition $\binom{n}{-2} = \binom{-2}{-1} = 0$ for all $n$. Checking this is again routine.
\end{proof}

%Since $\Mo{n}{x} = \Cat{n+1}{x,0}$ and $\MoB{n}{x} = \CatB{n}{x,0}$, setting $y=0$ in the preceding result yields a recurrence for these polynomials as well. 
%In particular, for $n\geq 2$ we have
%\be\label{end1}
%\ba (n+2)\Mo{n}{x} &=(2n+1)  \Mo{n-1}{x} + (4x-1)(n-1)  \Mo{n-2}{x},
%\\
%n \MoB{n}{x} &= (2n-1) \MoB{n-1}{x} + (4x-1)(n-1)\MoB{n-2}{x}.
%\ea\ee
%%
%%
%The entry for sequence A055151 in \cite{OEIS} mentions the first recurrence, which Woan proves combinatorially in the special case $x=1$ \cite{MotzkinRec1}. Proofs of the recurrence for the central trinomial coefficients given by setting $x=1$ in the second equation have appeared in many places; Sulanke \cite{Sulanke_moments} even notes a reference for this in one of Euler's papers.

\section{Connections to representation theory}
\label{last-sect}

\def\UTB{\UT^{\mathrm{B}}}
\def\UTD{\UT^{\mathrm{D}}}

Fix a finite field $\FF_q$ with with order $q$ and \emph{odd characteristic} $p > 0$.
Given positive integers $m,n$, we write $\FF_q^{m\times n}$ for the set of $m\times n$ matrices over $\FF_q$. %, and let $\FF_q^n = \FF_q^{n\times 1}$.  
For any matrix $X \in \FF_q^{m\times n}$, 
let $X^\dag \in \FF_q^{n\times m}$ denote the backwards transposed matrix defined by $(X^\dag)_{ij} = X_{n+1-j,m+1-i}$.

%We define three finite groups $\UT_n(\FF_q)$, $\UTB_n(\FF_q)$, and $\UTD_n(\FF_q)$ as follows.
For each positive integer $n$, let $\UT_n(\FF_q)$ denote the group of $n\times n$ upper triangular matrices over $\FF_q$ whose diagonal entries are all equal to one, and 
 let $\UTB_n(\FF_q)\subset \UT_{2n+1}(\FF_q)$ and $\UTD_n(\FF_q)\subset \UT_{2n}(\FF_q)$ be the subgroups of elements fixed by the involution $g \mapsto (g^{-1})^\dag$.  Explicitly,
this means:
\[
\ba 
\UT_n(\FF_q) &=\Bigl\{ x \in \FF_q^{n\times n} : x_{ii} = 1\text{ and } x_{ji} = 0\text{ for }i,j \in [n] \text{ with }i<j\Bigr\},
\\
 \UTB_n(\FF_q)&= \left\{ 
\(\barr{lll} x & xu & xz \\ 0 & 1 & -u^\dag \\ 0 & 0 & (x^{-1})^\dag \earr\) : x \in \UT_n(\FF_q),\ u \in \FF_q^{n\times 1},\ z \in \FF_q^{n\times n},\ z^\dag + z +uu^\dag=0
\right\},
\\
\UTD_n(\FF_q)&= \left\{ 
\(\barr{ll} x & xz \\  0 &  (x^{-1})^\dag \earr\) : x \in \UT_n(\FF_q),\ z \in \FF_q^{n\times n},\ z^\dag + z = 0
\right\}.
 \ea\]
% \end{enumerate}
The groups $\UT_{n+1}(\FF_q)$, $\UTB_n(\FF_q)$, and $\UTD_n(\FF_q)$ are isomorphic to the Sylow $p$-subgroups of the 
Chevalley groups $A_n(q)$, $B_n(q)$, and $D_n(q)$.

Finding a general classification of the irreducible representations of these and related groups for all $n$ and $q$ is a well-known wild problem.  
However, in the past two decades
a series of researchers led by C. A. M. Andr\'e have defined and studied useful \emph{supercharacter theories} for these groups, with many notable combinatorial properties. % closely related to the  set partitions combinatorics described in this work.
Introduced by Diaconis and Isaacs \cite{DI}, a supercharacter theory of a finite group $G$ is a set $\cS$ of complex characters  of $G$, called \emph{supercharacters}, such that (i) every irreducible character of $G$ is a constituent of exactly one character $\chi \in \cS$; and (ii) $\cS$ has the same cardinality as $\cK$, the coarsest partition of $G$ with the property that each character $\chi \in \cS$ is constant on each set $K \in \cK$.  This definition leads naturally to the idea of a ``supercharacter table'' and other analogues of notions from character theory.

\def\halve{\emph{halve}}

While the groups $\UT_n(\FF_q)$, $\UTB(\FF_q)$, and $\UTD(\FF_q)$ in fact have many different supercharacter theories, each has one in particular whose supercharacters are naturally indexed by $\FF_q$-labeled set partitions of an appropriate type.  We  endeavor here to briefly describe these supercharacter theories, with the goal of lending representation theoretic meanings to some of the objects examined in the previous sections.

For a reference to the following material, see \cite{Thiem}. We may define the supercharacters of $\UT_n(\FF_q)$ by an explicit formula. Fix a nontrivial homomorphism $\theta : \FF_q \to \CC^\times$ from the additive group of the field to the complex numbers.  There is an equivalence relation on $\UT_n(\FF_q)$ defined by setting $x \sim y$ if and only if there are $g,h \in \UT_n(\FF_q)$ with $g(x-1)h = y-1$.  The matrices 
\[ x_\Gamma \omdef = 1 + \sum_{(i,j) \in \Arc(\Lambda)} \Lambda_{ij} e_{ij} \in\UT_n(\FF_q),\qquad\text{for }\Gamma \in \sP(n,\FF_q),\] with $e_{ij}$ denoting an $n\times n$ elementary matrix, are representatives of the equivalence classes of this relation, which we call \emph{superclasses}.
For each $\Lambda \in \sP(n,\FF_q)$, let $\chi_\Lambda : \UT_n(\FF_q) \to \CC$ be the function which is constant on superclasses with the formula
\be \chi_\Lambda(x_\Gamma) = \begin{cases}
\displaystyle \prod_{(i,l) \in \Arc(\Lambda)} q^{c(i,l;\Gamma)}  \theta\(\Lambda_{ij}\Gamma_{ij}\) ,&\barr{l} \text{if }\{(i,j),(j,k) \} \cap \Arc(\Gamma)= \varnothing\text{ for all $i,j,k$} \\ \text{with $i<j<k$ and $(i,k) \in \Arc(\Lambda)$,}\earr \\ 0,&\text{otherwise,}\end{cases}\ee for $\Gamma \in \sP(n,\FF_q)$, 
where $c(i,l;\Gamma) \omdef=l-i-1- |\{ (j,k) \in \Arc(\Gamma) : i<j<k<l\}|$ and where $\Gamma_{ij}$ is defined to be zero if $(i,j) \notin \Arc(\Gamma)$.
%While it is not clear \emph{a priori} that the functions $\chi_\Lambda$ are characters, the set $\{ \chi_\Lambda \}$ will form a supercharacter theory of $\UT_n(\FF_q)$

In the sequence of papers \cite{bcd1,bcd2,bcd3}, Andr\'e and Neto introduce an analogous set of supercharacters for the groups $\UTB_n(\FF_q)$ and $\UTD_n(\FF_q)$.  One may succinctly define these  as  restrictions of certain $\chi_\Lambda$'s.   
To state this clearly, 
let $\halve(\Lambda)$ for $\Lambda\in \sPB(n,\FF_q)$ be the labeled set partition in $\sP(2n+1,\FF_q)$
given by removing all arcs $(i,j)$ with $i+j > 0$ from $\Lambda$ and then applying the natural embedding of $\FF_q$-labeled partitions of $[\pm n] \cup \{0\}$ in $\sP(2n+1,\FF_q)$.
%Explicitly, $\halve(\Lambda)$ is the element of $\sP(2n+1,\FF_q)$ with arc set 
%\[\Arc\(\halve(\Lambda)\) = \{ (i+n+1,j+n+1)  : (i,j) \in \Arc(\Lambda)\text{ and }i+j\leq 0\}\] and labeling map $(i+n+1,j+n+1) \mapsto \Lambda_{ij}$.
We define $\halve(\Gamma) \in \sP(2n,\FF_q)$ for $\Gamma \in \sPD(n,\FF_q)$
by the same operation; i.e., remove all arcs $(i,j)$ with $i+j > 0$ from $\Lambda$ and then apply the natural embedding of $\FF_q$-labeled partitions of $[\pm n] $ in $\sP(2n,\FF_q)$.
%this amounts to the slightly more intricate explicit definition that $\halve(\Gamma)$ for $\Gamma \in \sPD(n,\FF_q)$ is the element of $\sP(2n,\FF_q)$ with arc set 
%\[\Arc\(\halve(\Gamma)\) =\{ (i+n+1,j+n+\varepsilon_j)  : (i,j) \in \Arc(\Lambda)\text{ and }i+j\leq 0\}\]  and labeling map $(i+n+1,j+n+\varepsilon_j) \mapsto \Gamma_{ij}$, where $\varepsilon_j$ is  1 if $j<0$ and 0 otherwise.
For example,
\[ \halve\( \xy<0.0cm,-0.0cm> \xymatrix@R=-0.0cm@C=.1cm{
*{\bullet} \ar @/^0.5pc/ @{-} [r]^a &
*{\bullet} \ar @/^.9pc/ @{-} [rr]^b &
*{\bullet} \ar @/^.9pc/ @{-} [rr]^{-b}  &
*{\bullet}   &
*{\bullet} \ar @/^0.5pc/ @{-} [r]^{-a}&
*{\bullet} 
\\
-3 &
-2   & 
-1 &
+ 1 &
+ 2 &
+ 3 
}\endxy\)
=
 \xy<0.0cm,-0.0cm> \xymatrix@R=-0.0cm@C=.1cm{
*{\bullet} \ar @/^0.5pc/ @{-} [r]^a &
*{\bullet} \ar @/^.9pc/ @{-} [rr]^b &
*{\bullet}  &
*{\bullet}   &
*{\bullet} &
*{\bullet} 
\\
1 &
2   & 
3 &
4 &
5 &
6 
}\endxy \]
for $a,b \in \FF_q^\times$.  Now, for $\Lambda \in \sPB(n,\FF_q)$ and $\Gamma \in \sPD(n,\FF_q)$ we let $\chiB_\Lambda : \UTB_n(\FF_q) \to \CC$ and $\chiD_\Gamma : \UTD_n(\FF_q) \to \CC$ denote the restricted functions
\[ \chiB_\Lambda \omdef= \Res_{\UTB_n(\FF_q)}^{\UT_{2n+1}(\FF_q)} \( \chi_{\halve(\Lambda)}
\)
\qquad\text{and}\qquad
 \chiD_\Gamma\omdef = \Res_{\UTD_n(\FF_q)}^{\UT_{2n}(\FF_q)} \( \chi_{\halve(\Gamma)}
\).
\]
The significance of these functions derives from the following theorem.

\newcommand{\SCT}[2]{SCT^{\mathrm{#1}}_{#2}}

\begin{theorem}\label{ps}
When $q$ is odd, the sets  of functions
\[ \{ \chi_\Lambda : \Lambda \in \sP(n,\FF_q)\}, \quad 
\{ \chiB_\Lambda : \Lambda \in \sPB(n,\FF_q)\},\quad 
\text{and}\quad \{ \chiD_\Lambda : \Lambda \in \sPD(n,\FF_q)\}\] are supercharacter theories of the groups $\UT_n(\FF_q)$, $\UTB_n(\FF_q)$, and $\UTD_n(\FF_q)$. These theories have $\Bell{n}{q-1}$, $\BellB{n}{q-1}$,  and $\Bell{n}{2q-2}$ distinct supercharacters, respectively.
\end{theorem}

 In particular, this means the functions $\chi_\Lambda$, $\chiB_\Lambda$, and $\chiD_\Lambda$ are characters. When $q$ is even, $\{ \chi_\Lambda\}$ is still a supercharacter theory of $\UT_n(\FF_q)$. The methods used to show that the second two sets are supercharacter theories, however, depend on odd characteristic.

\begin{proof}
That $\{ \chi_\Lambda : \Lambda \in \sP(n,\FF_q)\}$ is a supercharacter theory of $\UT_n(\FF_q)$ is well-known; see \cite{Thiem}. 
Our definition of the supercharacters
$\{ \chiB_\Lambda : \Lambda \in \sPB(n,\FF_q)\}$ and $\{ \chiD_\Lambda : \Lambda \in \sPD(n,\FF_q)\}$ 
differs from the one given by Andr\'e and Neto in \cite{bcd1,bcd2,bcd3}, but is equivalent by \cite[Proposition 2.2]{bcd2}. Hence it follows from the main conclusions of 
 \cite{bcd1,bcd2,bcd3} that both sets are supercharacter theories.
\end{proof}

The formula for $\chi_\Lambda$ shows that  $\Lin(n,\FF_q)$, $\LinB(n,\FF_q)$, and $\LinD(n,\FF_q)$  index the linear supercharacters (i.e., those with degree 1) of $\UT_n(\FF_q)$, $\UTB(n,\FF_q)$, and $\UTD(n,\FF_q)$, respectively.  These supercharacters exhaust the set of all linear characters for $\UT_n(\FF_q)$ and $\UTB_n(\FF_q)$, but not for $\UTD_n(\FF_q)$.  In addition, it is not difficult to check that the product of a supercharacter with a linear supercharacter remains a supercharacter, and that, as mentioned in the introduction, we have 
\[ \chi_\alpha \chi_\Lambda = \chi_{\alpha+\Lambda},
\qquad
\chiB_{\alpha'}\chiB_{\Lambda'}= \chiB_{\alpha'+\Lambda'},
\qquad\text{and}\qquad
\chiD_{\alpha''}\chiD_{\Lambda''} = \chiD_{\alpha''+\Lambda''}\]
for $\alpha \in \Lin(n,\FF_q)$, $\alpha'\in\LinB(n,\FF_q)$, $\alpha''\in \LinD(n,\FF_q)$ and $\Lambda \in \sP(n,\FF_q)$, $\Lambda'\in\sPB(n,\FF_q)$, $\Lambda''\in\sPD(n,\FF_q)$. As a consequence, we may state the following.

\begin{corollary} The supercharacter theories of $\UT_n(\FF_q)$, $\UTB_n(\FF_q)$, and $\UTD_n(\FF_q)$ in Theorem \ref{ps}
%the groups $\UT_n(\FF_q)$, $\UTB_n(\FF_q)$,  $\UTD_n(\FF_q)$
 have  
%$\Bell{n}{q-1}$, $\BellB{n}{q-1}$,  $\Bell{n}{2q-2}$ distinct supercharacters and 
$\Fe{n-1}{q-1}$, $\Fe{n}{2q-2}$, and $\wFeB{n-1}{q-1}$  supercharacters, respectively, which are invariant under multiplication by all linear supercharacters.
\end{corollary}

We conclude with a result which gives a representation theoretic interpretation of the sets $\tNCB(n,\AA)$ and $\tNCD(n,\AA)$ introduced in Section \ref{b-section}.

%\item[(3)] $\Mo{n-1}{q-1}$, $\MoB{n}{q-1}$,  $\wMoB{n-1}{q-1}$  orbits of supercharacter under the action of  the multiplicative group of linear supercharacters. 

\begin{theorem}\label{last-thm}
With respect to the supercharacter theories in Theorem \ref{ps}, the following properties hold:

\begin{enumerate}
\item[(1)] The map $\Lambda \mapsto \chi_\Lambda$ defines a bijection from $\NC(n,\FF_q)$ to the set of irreducible supercharacters of $\UT_n(\FF_q)$.
More generally, the groups $\UT_n(\FF_q)$, $\UTB_{n-1}(\FF_q)$,  $\UTD_n(\FF_q)$ all have exactly $\Cat{n}{q-1}$ distinct irreducible supercharacters.

\item[(2)] The map $\Lambda \mapsto \chiB_\Lambda$ defines a bijection from $\tNCB(n,\FF_q)$ to the set of supercharacters of $\UTB_n(\FF_q)$ not equal to the restriction of any reducible supercharacter of $\UT_{2n+1}(\FF_q)$.

\item[(3)] The map $\Lambda \mapsto \chiD_\Lambda$ defines a bijection from $\tNCD(n,\FF_q)$ to the set of supercharacters of $\UTD_n(\FF_q)$  not equal to the restriction of any reducible supercharacter of $\UT_{2n}(\FF_q)$.

\end{enumerate}
\end{theorem}

\begin{proof} 
The first sentence in part (1) is noted in several places; e.g.,  \cite{Thiem}. To show that $\Cat{n}{q-1}$ also counts the irreducible supercharacters of $\UTB_{n+1}(\FF_q)$ and $\UTD_n(\FF_q)$, it suffices to show that $\chiB_\Lambda$ is irreducible if and only if $\Lambda \in \sPB(n,\FF_q)$ is noncrossing and 
$\chiD_\Lambda$ is irreducible if and only if $\Lambda \in \sPD(n,\FF_q)$ is noncrossing.
We  only address the type B case as the proof in type D is similar.

Given integers $-n \leq i<j \leq n$ with $i\neq -j$ and $t \in \FF_q^\times$, let $\chiB_{i,j,t}$ denote the supercharacter of $\UTB(n,\FF_q)$ indexed by the labeled set partition in $\sPB(n,\FF_q)$ with just two arcs $(i,j),(-j,-i)$ labeled by $t,-t$.
From the original construction in \cite[Section 2]{bcd3}, it follows that $\chiB_\Lambda = \prod_{i,j,t} \chiB_{i,j,t}$, where $(i,j)$ ranges over all  $(i,j)\in \Arc(\Lambda)$ with $i+j \leq 0$ and $t \in \FF_q^\times$ is the associated label.  Moreover, if $0\leq j<i\leq n$ and 
\[ H = \left \{ g \in \UTB(n,\FF_q) : \ba & g_{n+1-i,k}= 0\text{ if } n+1-i < k \leq n,
\\
& g_{n+1-j,l} = 0\text{ if } n+1-j < l \leq n+1
\ea
\right\}\]
then, as defined in \cite{bcd3}, $\chiB_{-i,j,t}$ is the character induced from the linear character $\tau$ of $H$ defined by $\tau : g \mapsto \theta(tg_{n+1-i,n+1+j})$.
In this situation, one checks that the formula for $\tau$ also defines a linear character of the group 
\[ K = \left \{ g \in \UTB(n,\FF_q) : \ba & g_{n+1-i,k}= 0\text{ if } n+1-i < k \leq n \text{ and } k\neq n+1-j,
\\
& g_{n+1-j,l} = 0\text{ if } n+1-j < l \leq n+1
\ea
\right\} \supsetneq H \]
whence it follows by Frobenius reciprocity and Mackey's theorem that $\chiB_{-i,j,t}$ has a proper constituent (given by inducing $\tau$ from $K$).  This discussion shows that if
$\Lambda \in \sPB(n,\FF_q)$ has an arc of the form $(-i,j)$ where $i,j > 0$,
then $\chiB_\Lambda$ is reducible.
%A similar argument shows  that if $\Lambda \in \sPD(n,\FF_q)$ has an arc of the form $(-i,j)$ where $i,j>0$, then $\chiD_\Lambda$ is reducible.

If $\halve(\Lambda)$ has any crossing arcs then $\chiB_\Lambda$ is reducible since it is the restriction of a reducible character. It is straightforward to see that if $\Lambda \in \sPB(n,\FF_q)$ has no arcs $(-i,j)$ with $i,j>0$, then $\halve(\Lambda)$ is noncrossing if and only if $\Lambda$ is noncrossing.
In view of the previous paragraph, it follows that $\chiB_\Lambda$ is irreducible only if $\Lambda \in \sPB(n,\FF_q)$ is noncrossing. If $\Lambda$ is noncrossing then
 $\{n+i+1\}$ is a singleton block of $\halve(\Lambda)$ for each $i \in [n]$. In this case, define $\Lambda' \in \sP(n+1,\FF_q)$ as the labeled set partition formed from $\halve(\Lambda)$ by removing the blocks $\{n+i+1\}$, and let $N$ be the normal subgroup of matrices  $g\in \UTB_n(\FF_q)$ such that $g-1$ has all zeros in the first $n+1$ columns and in the last $n+1$ rows. The formula for $\chi_{\halve(\Lambda)}$ then shows that   $\chiB_\Lambda$ is  
obtained by inflating the supercharacter $\chi_{\Lambda'}$ of the quotient $ \UT_{n+1}(\FF_q) \cong \UTB(n,\FF_q) / N$. Since $\Lambda$ is noncrossing, $\Lambda'$ is noncrossing, so both $\chi_{\Lambda'}$ and $\chiB_\Lambda$ are irreducible.  It follows that $\chiB_\Lambda$ is irreducible if and only if $\Lambda \in \sPB(n,\FF_q)$ is noncrossing, as required. % so $\UTB(n,\FF_q)$ has $\Cat{n+1}{q-1}$ irreducible supercharacters since the noncrossing elements of $\sPB(n,\FF_q)$ are in bijection with $\NC(n+1,\FF_q)$.  %A similar argument shows that $\chiD_\Lambda$ is irreducible if and only if $\Lambda \in \sPD(n,\FF_q)$ is noncrossing, so $\UTD(n,\FF_q)$ has $\Cat{n}{q-1}$ irreducible supercharacters.  

%The proofs of (b) and (c) are similar, so we will only address the type B case. 
To prove (b), say that two partitions $\Lambda,\Lambda' \in \sP(2n+1,\FF_q)$ are equivalent if $\Lambda'$ can be obtained from $\Lambda$ by replacing an arc $(i,j) \in \Arc(\Lambda)$ labeled by $t$ with the arc $(2n+2-j,2n+2-i)$ labeled by $-t$; here, the arc $(i,j)$ can be replaced only if $(2n+2-j,2n+2-i)$ does not already belong to $\Arc(\Lambda)$.
If we extend this to define an equivalence relation on $\sP(2n+1,\FF_q)$, then it follows from the combination of \cite[Proposition 2.2]{bcd2} and  \cite[Corollary 4.7]{Thiem} that  a supercharacter $\chiB_\Gamma$ for $\Gamma \in \sPB(n,\FF_q)$  is equal to the restriction of $\chi_\Lambda$ if and only if $\Lambda \in \sP(2n+1,\FF_q)$ belongs to the equivalence class of $\halve(\Gamma)$.  In turn, it is easy to check that the equivalence class of $\halve(\Gamma)$ contains only noncrossing elements if and only if every pair of crossing arcs in $\Arc(\Gamma)$ has the form $(i,j),(-j,-i)$; i.e., if and only if $\Gamma \in \tNCB(n,\FF_q)$.  The proof of (c) is similar.
\end{proof}

\begin{corollary} The polynomials $\CatB{n}{q-1}$ 
and $\CatD{n}{q-1}$ count the  supercharacters of $\UTB_n(\FF_q)$ and $\UTD_n(\FF_q)$  which are not equal to the restriction of any reducible supercharacter of $\UT_{2n+1}(\FF_q)$ and $\UT_{2n}(\FF_q)$, respectively.
\end{corollary}

\end{document}